\documentclass[a4paper,11pt]{article}

\usepackage[margin=25mm]{geometry}
\usepackage[T1]{fontenc}
\usepackage{tgtermes}
\usepackage{tgheros}
\DeclareEmphSequence{\itshape,\upshape}

\usepackage{amsmath,amssymb,amsthm}
\numberwithin{equation}{section}
\theoremstyle{plain}
\newtheorem{theorem}{Theorem}[section]
\newtheorem{proposition}[theorem]{Proposition}
\newtheorem{lemma}[theorem]{Lemma}
\newtheorem{corollary}[theorem]{Corollary}
\newtheorem{maintheorem}{Theorem}
\newtheorem{maincorollary}[maintheorem]{Corollary}

\theoremstyle{definition}
\newtheorem{definition}[theorem]{Definition}
\theoremstyle{remark}
\newtheorem{remark}[theorem]{Remark}

\usepackage{enumitem}
\setlist[enumerate]{label=\textup{(\arabic*)}}

\usepackage{array,booktabs,longtable}
\usepackage{tikz}
\usepackage[section]{placeins}
\usetikzlibrary{arrows.meta}

\usepackage{xcolor}
\usepackage[hyperfootnotes=false]{hyperref}
\hypersetup{
  unicode=true,
  colorlinks=true,
  linkcolor=blue!50!black,
  citecolor=green!35!black,
  urlcolor=blue!60!black,
  pdftitle={Rotation fixed points of undecorated braid varieties and fusion rings of affine Lie algebras},
  pdfauthor={Yuma Mizuno}
}
\usepackage[
  backend=biber,
  style=numeric-comp,
  maxbibnames=99,
  giveninits=true,
  isbn=false,
]{biblatex}

\DeclareFieldFormat*{title}{\mkbibemph{#1}}
\DeclareFieldFormat{journaltitle}{#1}
\DeclareFieldFormat{booktitle}{#1}
\renewbibmacro{in:}{\ifentrytype{article}{}{\printtext{\bibstring{in}\intitlepunct}}}
\DeclareMathOperator{\Spec}{Spec}
\DeclareMathOperator{\Conf}{Conf}
\DeclareMathOperator{\rot}{rot}
\DeclareMathOperator{\pos}{pos}
\DeclareMathOperator{\Nm}{Nm}
\DeclareMathOperator{\Ad}{Ad}
\DeclareMathOperator{\id}{id}
\DeclareMathOperator{\len}{len}
\newcommand{\Z}{\mathbb Z}
\newcommand{\Q}{\mathbb Q}
\newcommand{\C}{\mathbb C}
\newcommand{\Dcal}{\mathcal D}
\newcommand{\Aff}{\operatorname{Aff}}
\newcommand{\Bcal}{\mathcal B}
\newcommand{\BZ}{\mathrm BZ^\theta}

\title{Rotation fixed points of undecorated braid varieties\\ and fusion rings of affine Lie algebras}
\author{Yuma Mizuno}
\usepackage{titling}
\predate{}\postdate{}
\date{}
\makeatletter
\newcommand{\titlenote}[1]{\begingroup
  \renewcommand{\thefootnote}{}%
  \renewcommand{\@makefntext}[1]{\parindent=1em ##1}%
  \footnotetext{#1}\endgroup}
\makeatother

\begin{document}
\maketitle

\begin{abstract}
Let \(G\) be a simply connected semisimple algebraic group.
On the (undecorated) braid variety associated with a power of a Coxeter element,
there is a finite-order autoequivalence, called the Zamolodchikov transformation, induced by cyclic rotation of words.
In this paper, we show that the coarse moduli space of the fixed-point stack of this autoequivalence is isomorphic to the spectrum of the
fusion ring of the corresponding affine Lie algebra.
The above results also extend to twisted Coxeter elements, and the fusion rings of all affine Lie algebras
are obtained in this way.
\end{abstract}

\section{Introduction}
\label{sec:main}

\subsection{Background}
The integrable highest weight representations of level \(\ell\) of an untwisted affine Lie algebra
are basic components of the state spaces of the Wess--Zumino--Witten models of two-dimensional conformal field theory,
and \(\mathrm{SL}_2(\Z)\) acts on the space spanned by their normalized characters
\cite{KacPeterson1984}.
The fusion ring of level \(\ell\) is the commutative ring that has the isomorphism classes of these representations as a basis and the fusion rules of the conformal field theory as structure constants.
Verlinde found that the \(S\)-matrix of the action of \(\mathrm{SL}_2(\Z)\) diagonalizes the fusion rules \cite{Verlinde1988}.
Mathematically, the fusion rules are formulated as the dimensions of conformal blocks on the projective line with three marked points \cite{TsuchiyaUenoYamada1989},
and their diagonalization yields the Verlinde formula, which gives the dimensions of conformal blocks on general curves
\cite{Faltings1994,Teleman1995,Beauville1996,Teleman1996,Huang2008}.

The central charges of the conformal field theories associated with affine Lie algebras are related through the Rogers dilogarithm
to \(Y\)-systems, which are systems of difference equations of integrable models.
Zamolodchikov obtained \(Y\)-systems from the thermodynamic Bethe ansatz equations of scattering theories of type \(ADE\),
and conjectured the periodicity of their solutions \cite{Zamolodchikov1991}.
This conjecture was generalized to the periodicity of the solutions of the level restricted \(Y\)-system determined by each pair of an untwisted affine Lie algebra and a level \(\ell\) \cite{KunibaNakanishi1992,RavaniniTateoValleriani1993,KunibaNakanishiSuzuki1994}.
This periodicity was proved
using cluster algebras \cite{FominZelevinsky2002} and their additive categorification by triangulated categories \cite{Keller2013,IIKKN2013B,IIKKN2013CFG}.
The sum of the values of the Rogers dilogarithm at a constant solution of the level restricted \(Y\)-system consisting of positive real numbers
is expressed by the central charge of the parafermion conformal field theory of level \(\ell\) of the corresponding affine Lie algebra.
This dilogarithm identity was proved using the periodicity \cite{Nakanishi2011Dilog,IIKKN2013B,IIKKN2013CFG}.
More generally, a construction of constant solutions of the level restricted \(Y\)-system using character values of finite-dimensional representations
at torus elements determined by dominant integral weights of level \(\ell\) has been proposed \cite{KunibaNakanishi1992}.
It is conjectured that the sums of dilogarithms at these solutions, with logarithmic correction terms added,
are expressed by the central charge and the scaling dimensions of the parafermion conformal field theory \cite{KunibaNakanishi1992}.
From the relation between these constant solutions and the eigenvalues of the fusion matrices,
it was pointed out that the fusion ring is the algebraic structure underlying the dilogarithm identities \cite{NahmRecknagelTerhoeven1993}.

Today, the periodicity of \(Y\)-systems is interpreted geometrically as the periodicity of rotation automorphisms on braid varieties associated with powers of Coxeter elements.
A braid variety is a moduli space of sequences of flags whose relative positions change along a positive braid
\cite{ShendeTreumannZaslow2017,CasalsGorskyGorskySimental2024}.
Braid varieties form a class containing double Bruhat cells, double Bott--Samelson cells, and open Richardson varieties,
and they have cluster structures
\cite{CasalsEtAl2025,GalashinLamShermanBennettSpeyer2022,GalashinLamShermanBennett2026}.
Special cases of braid varieties had been introduced earlier by Shen--Weng under the name double Bott--Samelson cells,
and their cluster structures had been constructed \cite{ShenWeng2021}.
Shen--Weng further constructed a rotation automorphism of finite order on undecorated double Bott--Samelson cells,
and, since it gives the time evolution of the level restricted \(Y\)-systems of type \(ADE\) in cluster coordinates,
they gave a geometric proof of Zamolodchikov periodicity \cite{ShenWeng2021}.
In type \(A\), braid varieties appear as augmentation varieties of Legendrian links
\cite{CasalsGorskyGorskySimental2024,GaoShenWeng2024,AsplundEtAl2025},
and their cluster seeds are obtained from exact Lagrangian fillings
\cite{ShendeTreumannWilliamsZaslow2019,CasalsZaslow2022,CasalsWeng2024,GaoShenWeng2024,CasalsGao2024}.
In this case, the time evolution of the \(Y\)-system is the automorphism induced by the Kálmán loop of a Legendrian torus link \cite{GaoShenWeng2024},
and Kálmán had computed the order of the monodromy of this loop \cite{Kalman2005}.

In this paper, for the rotation automorphism on the undecorated braid variety associated with a power of a (twisted) Coxeter element,
we show that the coarse moduli space of the fixed-point stack is the spectrum of the fusion ring of an affine Lie algebra (Theorem~\ref{thm:main}).
When the affine Lie algebra is untwisted, this rotation automorphism gives the time evolution of the level restricted \(Y\)-system in cluster coordinates
(type \(ADE\) by Shen--Weng, type \(BCFG\) in Appendix~\ref{subsec:appendix-cluster}).
That is, the fusion ring expected to underlie the dilogarithm identities is realized as the fixed-point space of the rotation automorphism that gives the \(Y\)-system.

\subsection{Fixed-point sets and fusion rings: the case of type \texorpdfstring{\(A_1\)}{A1}}
\label{subsec:A1-intro}
Before stating the results in the general case, we explain the content of this paper in the case of type \(A_1\).
Let \(\ell\) be a natural number.
The cyclic configurations of \(\ell+2\) points of \(\mathbb P^1\) in which adjacent points are distinct
form an open subvariety \(C_{\ell+2}\) of \((\mathbb P^1)^{\ell+2}\). Its set of \(\C\)-points is
\[
 C_{\ell+2}(\C)=\{p=(p_j)_{j\in\Z/(\ell+2)\Z}\in(\mathbb P^1(\C))^{\ell+2}\mid p_j\ne p_{j+1}\}.
\]
We can represent each configuration by placing a point of \(\mathbb P^1(\C)\) at each vertex of a polygon with \(\ell+2\) vertices (Figure~\ref{fig:A1-cyclic}).

\begin{figure}[h]
 \centering
 \begin{tikzpicture}[scale=0.70]
  \foreach \j in {0,...,5} {
   \coordinate (v\j) at ({90-60*\j}:1.45);
   \fill (v\j) circle (1.6pt);
   \node at ({90-60*\j}:1.95) {\(p_\j\)};
  }
  \draw (v0)--(v1)--(v2)--(v3)--(v4)--(v5)--cycle;
  \draw[->] (35:1) arc[start angle=35,end angle=-35,radius=1];
  \node[font=\footnotesize] at (0,0) {rotation};
 \end{tikzpicture}
 \caption{A cyclic configuration with \(\ell=4\).}
 \label{fig:A1-cyclic}
\end{figure}
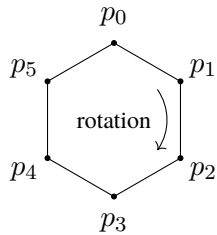

We write the quotient stack by the diagonal action of \(\mathrm{SL}_2\) as
\(\Conf^{\mathrm{cyc}}_{\ell+2}=[C_{\ell+2}/\mathrm{SL}_2]\).
Since the rotation \((p_j)_j\mapsto(p_{j-1})_j\) commutes with the diagonal action of \(\mathrm{SL}_2\),
it defines an autoequivalence \(\rot\) of \(\Conf^{\mathrm{cyc}}_{\ell+2}\).
We write \((\Conf^{\mathrm{cyc}}_{\ell+2})^{\rot}\) for the fixed-point stack of \(\rot\).
The set of isomorphism classes of its \(\C\)-points is described as follows:
\begin{equation}
 \begin{aligned}
 &(\Conf^{\mathrm{cyc}}_{\ell+2})^{\rot}(\C)/{\cong}\\
 &\quad=\{(p,\gamma)\in C_{\ell+2}(\C)\times\mathrm{SL}_2(\C)\mid
 \forall j\in\Z/(\ell+2)\Z,\ p_{j+1}=\gamma\cdot p_j\}
 /\mathrm{SL}_2(\C).
 \end{aligned}
 \label{eq:A1-fixed-points}
\end{equation}
Here \(g\in\mathrm{SL}_2(\C)\) acts from the left by
\(g\cdot(p,\gamma)=(g\cdot p,g\gamma g^{-1})\).

The fixed-point set \eqref{eq:A1-fixed-points} has the following explicit description.
Define polynomials \(P_j\in\Z[x]\) by
\(P_{-1}=0\), \(P_0=1\), \(P_{j+1}=xP_j-P_{j-1}\)
(in terms of the Chebyshev polynomials of the second kind \(U_j\), \(P_j(x)=U_j(x/2)\)).
Each element of the fixed-point set has a unique representative \((p,\gamma)\) such that
\[
 p_0=[1:0],\qquad \gamma=\begin{pmatrix}x&-1\\1&0\end{pmatrix}
\]
for some \(x\in\C\). In this case,
\[
 p_j=[P_j(x):P_{j-1}(x)]\qquad(0\le j\le\ell+1).
\]
The cyclic condition \(\gamma^{\ell+2}\cdot p_0=p_0\) is equivalent to \(P_{\ell+1}(x)=0\).
Therefore, via \(x=\operatorname{tr}\gamma\), the fixed-point set is in bijection with
the set \(\{2\cos(\pi a/(\ell+2))\mid 1\le a\le\ell+1\}\) of the \(\ell+1\) roots of \(P_{\ell+1}\).

If we write \(f_j\) for the function that assigns to an element \([(p,\gamma)]\) of the fixed-point set the trace of \(\gamma\) on the irreducible representation \(\mathrm{Sym}^j(\C^2)\) of \(\mathrm{SL}_2(\C)\),
then \(f_j=P_j(\operatorname{tr}\gamma)\).
In particular, the cyclic condition can be written as \(f_{\ell+1}=0\), and the following hold on the fixed-point set:
\[
 f_0=1,\qquad f_1f_j=f_{j-1}+f_{j+1}\quad(1\le j\le\ell),\qquad f_{\ell+1}=0.
\]
These are the relations of the fusion ring of type \(A_1^{(1)}\) at level \(\ell\).

Since \(P_j(2\cos\varphi)=\sin((j+1)\varphi)/\sin\varphi\), the value of \(f_j\) at the fixed point corresponding to \(a\) is
\[
 \frac{\sin((j+1)\pi a/(\ell+2))}{\sin(\pi a/(\ell+2))},
\]
which is the ratio \(S_{j,a-1}/S_{0,a-1}\) of entries of the \(S\)-matrix of type \(A_1^{(1)}\) at level \(\ell\) \cite[Example~13.8]{Kac1990}.

If we choose a triangulation of the polygon with \(\ell+2\) vertices, the cross-ratios of the four points around each diagonal form cluster coordinates.
In particular, the cross-ratios give solutions of the \(Y\)-system, and the rotation corresponds to its time evolution \cite{GliozziTateo1996}.
We add one remark.
If the points at the two ends of a diagonal coincide, the corresponding cross-ratio is either undefined or equal to \(-1\),
and no solution of the \(Y\)-system is obtained.
When \(\ell\ge2\) is even, the fixed point corresponding to \(a=(\ell+2)/2\) satisfies \(p_j=p_{j+2}\).
Since every triangulation contains, for some \(j\), the diagonal joining the vertices \(j-1\) and \(j+1\),
no constant solution of the \(Y\)-system is obtained from this fixed point, whichever triangulation is used (Figure~\ref{fig:A1-triangulations}).
Thus there is a rotation fixed point that cannot be described by a constant solution of the \(Y\)-system.
\begin{figure}[htbp]
 \centering
 \begin{tikzpicture}[scale=0.62]
  \foreach \xo in {0,5,10} {
   \foreach \v [count=\c from 0] in {\infty,0,\infty,0,\infty,0} {
    \coordinate (p\xo-\c) at ({\xo+1.45*cos(90-60*\c)},{1.45*sin(90-60*\c)});
    \fill (p\xo-\c) circle (1.6pt);
    \node at ({\xo+1.95*cos(90-60*\c)},{1.95*sin(90-60*\c)}) {\(\v\)};
   }
   \draw (p\xo-0)--(p\xo-1)--(p\xo-2)--(p\xo-3)--(p\xo-4)--(p\xo-5)--cycle;
  }
  \draw (p0-0)--(p0-3);
  \draw[dashed] (p0-0)--(p0-2) (p0-0)--(p0-4);
  \draw (p5-2)--(p5-5);
  \draw[dashed] (p5-0)--(p5-2) (p5-5)--(p5-3);
  \draw[dashed] (p10-0)--(p10-2)--(p10-4)--cycle;
 \end{tikzpicture}
 \caption{Triangulations at the fixed point with \(\ell=4\), \(a=3\). Dashed lines indicate diagonals at whose two ends the points are equal.}
 \label{fig:A1-triangulations}
\end{figure}
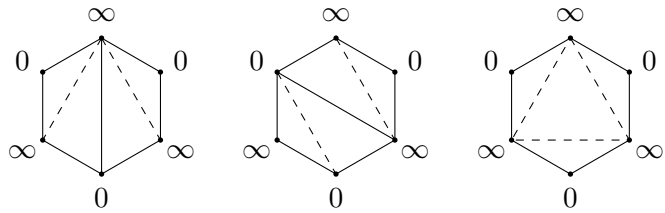

\subsection{Main results}

In what follows, let \(K\) be a field of characteristic \(0\).
Let \(G\) be a split simply connected semisimple group over \(K\), and assume that its Dynkin diagram is connected.
Fix a pinning \((B,T,\{x_i\}_{i\in I})\) of \(G\).
Here \(B\) is a Borel subgroup, \(T\subset B\) is a split maximal torus, \(I\) is the index set of simple roots, and
\(x_i\colon\mathbb G_a\to B\) is an isomorphism onto the root subgroup of the \(i\)-th simple root.
Let \(\theta\) be an automorphism of \(G\) preserving this pinning.
The automorphism \(\theta\) defines an automorphism \(\theta:I\to I\) of the Dynkin diagram of \(G\) by \(\theta(x_i(a))=x_{\theta(i)}(a)\).
Set \(d=\operatorname{ord}(\theta)\), and write \(Z^\theta\) for the \(\theta\)-fixed subgroup of \(Z(G)\) and \(\BZ\) for its classifying stack.
Let \(W=N_G(T)/T\) be the Weyl group, and write \(s_i\) for the simple reflection corresponding to \(i\in I\).
We call an element of \(W\) obtained by choosing one simple reflection from each \(\theta\)-orbit and multiplying them in an arbitrary order
a \emph{\(\theta\)-Coxeter element}.
We write \(\operatorname{Br}^+\) for the \emph{positive braid monoid} obtained by imposing the braid relations on words in \(I\).
The positive braid represented by a reduced word of \(w\in W\) does not depend on the choice of the reduced word, and we also denote it by \(w\).
The automorphism \(\theta\) acts on each letter of a word and defines an automorphism of \(\operatorname{Br}^+\).
Let \(\Delta\) be the positive braid of the longest element \(w_0\in W\).
For \(\ell\in\Z_{\ge0}\) and a \(\theta\)-Coxeter element \(c\), set
\[
 \beta=c\,\theta(c)\cdots\theta^{d-1}(c).
\]
The braid variety \(X(\beta^\ell\Delta)\) is the variety of sequences of flags that start at \(B\), end at \(w_0B\), and
have relative positions along the word \(\beta^\ell\Delta\)
\cite[\S3.3]{CasalsEtAl2025}.
The torus \(T\) acts on the sequences of flags by left multiplication on all flags.
Define \(i^*\) by \(w_0s_iw_0=s_{i^*}\), and set \((i_1\cdots i_m)^*=i_1^*\cdots i_m^*\) for a word \(i_1\cdots i_m\).
We call the transformation \(i_1i_2\cdots i_m\mapsto i_m^*i_1\cdots i_{m-1}\) of words
the rightward \emph{cyclic rotation}.
When the Demazure products of the words before and after the transformation are both \(w_0\), this operation induces an isomorphism
\[
 X(i_1i_2\cdots i_m)\xrightarrow{\sim}X(i_m^*i_1i_2\cdots i_{m-1})
\]
\cite[Lemma~3.10]{CasalsEtAl2025}.
If we extend a reduced word of \(c\) and write \(\Delta=ca=ac^*\), then the cyclic rotations give an automorphism
\[
 \rot:X(\beta^\ell\Delta)\xrightarrow{\theta}X(\theta(\beta)^\ell\Delta)
 =X(\theta(\beta)^\ell ac^*)
 \xrightarrow{\sim}X(c\theta(\beta)^\ell a)
 =X(\beta^\ell ca)=X(\beta^\ell\Delta).
\]
Since this \(\rot\) transforms the action of \(T\) by \(c\theta\),
it induces an autoequivalence of the quotient stack \([X(\beta^\ell\Delta)/T]\).
We also denote it by \(\rot\).
As shown in Appendix~\ref{app:rotation}, this rotation gives the time evolution of Zamolodchikov's \(Y\)-system through the cluster structure and is therefore also called the Zamolodchikov transformation \cite{ShenWeng2021,Keller2013}.
Setting \(h^\vee=\operatorname{ord}(c\theta)/d\), we have \(\rot^{d(\ell+h^\vee)}\cong\id\) as autoequivalences of \([X(\beta^\ell\Delta)/T]\) (Corollary~\ref{cor:appendix-periodicity}).
The fixed-point stack \([X(\beta^\ell\Delta)/T]^{\rot}\) of this autoequivalence consists of
pairs of an object \(\xi\) of \([X(\beta^\ell\Delta)/T]\) and an isomorphism \(\rot\xi\xrightarrow{\sim}\xi\)
(Definition~\ref{def:fixed}).

For a connected Dynkin diagram \(\Dcal\) and an automorphism \(\theta\) of it,
the \(\theta\)-twisted loop algebra of the complex simple Lie algebra with Dynkin diagram \(\Dcal\)
gives a realization of an affine Lie algebra \cite[Theorems~7.4, 8.3]{Kac1990}.
We write \(\Aff_\theta(\Dcal)\) for its affine Dynkin diagram.
The vertex set of \(\Aff_\theta(\Dcal)\) is obtained by adding a vertex \(0\) to the set of orbits of the vertices of \(\Dcal\) under \(\theta\)
\cite[\S7.4, \S8.3]{Kac1990}.
We denote the dual of a Dynkin diagram by \({}^\vee\), and we also write \(\theta\) for the induced automorphism of the dual diagram \(\Dcal^\vee\).
For an affine Dynkin diagram \(\mathsf A\) and \(\ell\in\Z_{\ge0}\),
we write \(P_\ell(\mathsf A)\) for the set of dominant integral weights of level \(\ell\) (Definition~\ref{def:level-weights}).
The fusion ring \(\mathsf R_\ell(\mathsf A)\) of level \(\ell\) of the affine Lie algebra associated with \(\mathsf A\) is
the commutative ring with \(\Z\)-basis the functions \(\chi_\mu\) obtained by restricting the Weyl characters with highest weight \(\mu\in P_\ell(\mathsf A)\) to a finite subset of the torus
(\cite[\S2]{Hong2018}, Definition~\ref{def:fusion}).
Since \(\mathsf R_\ell(\mathsf A)\otimes_\Z\C\) is isomorphic to the ring of functions on the quotient of this finite set by the Weyl group (Section~\ref{subsec:fusion-def}),
\(\mathsf R_\ell(\mathsf A)\otimes_\Z K\) is a finite étale \(K\)-algebra.

\begin{maintheorem}[{proved in Section~\ref{subsec:fusion-comparison}}]
\label{thm:main}
Let \(\Dcal\) be the Dynkin diagram of \(G\).
For every \(\ell\in\Z_{\ge0}\) and every \(\theta\)-Coxeter element \(c\), there is an equivalence
\[
 [X(\beta^\ell\Delta)/T]^{\rot}
 \simeq \BZ\times_K
 \Spec_K\bigl(\mathsf R_\ell(\Aff_\theta(\Dcal^\vee)^\vee)\otimes_\Z K\bigr).
\]
\end{maintheorem}

In particular, the stabilizer group of each geometric point of the fixed-point stack is \(Z^\theta\), and the coarse moduli space is
\[
 \Spec_K\bigl(\mathsf R_\ell(\Aff_\theta(\Dcal^\vee)^\vee)\otimes_\Z K\bigr).
\]

The number \(h^\vee\) is equal to the dual Coxeter number of \(\Aff_\theta(\Dcal^\vee)^\vee\) (Lemma~\ref{lem:appendix-full-twist}).
Table~\ref{tab:types} shows, for each type of \(\Dcal\) and \(\theta\), \(\Aff_\theta(\Dcal^\vee)^\vee\), its dual Kac labels \(a^\vee\) (Section~\ref{sec:fusion}),
\(h^\vee\), and \(Z^\theta\).

\begin{table}[t]
\centering
\renewcommand{\arraystretch}{1.15}
\begin{tabular}{@{}llllll@{}}
\toprule
\(\Dcal\) & \(\theta\) & \(\Aff_\theta(\Dcal^\vee)^\vee\) & \(h^\vee\) & \(a^\vee\) & \(Z^\theta\) \\
\midrule
\(A_r\) & \(\id\) & \(A_r^{(1)}\) & \(r+1\) & \(1^{r+1}\) & \(\mu_{r+1}\) \\
\(B_r\) & \(\id\) & \(D_{r+1}^{(2)}\) & \(2r\) & \(1^2,2^{r-1}\) & \(\mu_2\) \\
\(C_r\) & \(\id\) & \(A_{2r-1}^{(2)}\) & \(2r\) & \(1^2,2^{r-1}\) & \(\mu_2\) \\
\(D_r\) & \(\id\) & \(D_r^{(1)}\) & \(2r-2\) & \(1^4,2^{r-3}\) & \(\mu_2^2,\mu_4\) \\
\(E_6\) & \(\id\) & \(E_6^{(1)}\) & \(12\) & \(1^3,2^3,3\) & \(\mu_3\) \\
\(E_7\) & \(\id\) & \(E_7^{(1)}\) & \(18\) & \(1^2,2^3,3^2,4\) & \(\mu_2\) \\
\(E_8\) & \(\id\) & \(E_8^{(1)}\) & \(30\) & \(1,2^2,3^2,4^2,5,6\) & \(1\) \\
\(F_4\) & \(\id\) & \(E_6^{(2)}\) & \(12\) & \(1,2^2,3,4\) & \(1\) \\
\(G_2\) & \(\id\) & \(D_4^{(3)}\) & \(6\) & \(1,2,3\) & \(1\) \\
\midrule
\(A_{2r-1}\) & order \(2\) & \(B_r^{(1)}\) & \(2r-1\) & \(1^3,2^{r-2}\) & \(\mu_2\) \\
\(A_{2r}\) & order \(2\) & \(A_{2r}^{(2)}\) & \(2r+1\) & \(1,2^r\) & \(1\) \\
\(D_{r+1}\) & order \(2\) & \(C_r^{(1)}\) & \(r+1\) & \(1^{r+1}\) & \(\mu_2\) \\
\(E_6\) & order \(2\) & \(F_4^{(1)}\) & \(9\) & \(1^2,2^2,3\) & \(1\) \\
\(D_4\) & order \(3\) & \(G_2^{(1)}\) & \(4\) & \(1^2,2\) & \(1\) \\
\bottomrule
\end{tabular}
\caption{Data for each type (for the notation of affine types, see \cite[Tables~Aff~1--3]{Kac1990}).
The exponents in \(a^\vee\) denote multiplicities.
For \(D_r\), \(Z^\theta\) is \(\mu_2^2\) if \(r\) is even, and \(\mu_4\) if it is odd.}
\label{tab:types}
\end{table}

When \(K=\C\), we describe the set of isomorphism classes of \(\C\)-points of \([X(\beta^\ell\Delta)/T]^{\rot}\).
The stack \([X(\beta^\ell\Delta)/T]\) can be identified, by an equivalence compatible with the rotations, with the stack of cyclic flag configurations
\(\Conf^{\mathrm{cyc}}_{c,d(\ell+h^\vee)}\) (Lemma~\ref{lem:appendix-cut-rotation}).
In this presentation, the geometric points of the fixed-point stack are represented by pairs \((F,g)\) of a cyclic flag configuration \(F\) and
an element \(g\in G\) satisfying \(F=g\cdot\rot(F)\) (Section~\ref{sec:fixed}).
We regard the weights of the finite parts of \(\Aff_\theta(\Dcal^\vee)^\vee\) and \(\Aff_\theta(\Dcal)^\vee\)
as \(\theta\)-invariant weights and coweights of \(G\), respectively (Section~\ref{sec:fusion}).
For a \(\theta\)-invariant dominant weight \(\mu\), the irreducible representation \(V_\mu\) of \(G\) with highest weight \(\mu\)
extends to a representation of \(G\rtimes\langle\theta\rangle\) in which \(\theta\) fixes a highest weight vector (Section~\ref{sec:fusion}).
The rows and columns of the \(S\)-matrix \((S_{\mu\lambda})\) of level \(\ell\) \cite[Theorems~13.8, 13.9]{Kac1990} of the affine Lie algebra associated with \(\Aff_\theta(\Dcal^\vee)^\vee\)
are indexed by \(P_\ell(\Aff_\theta(\Dcal^\vee)^\vee)\) and \(P_\ell(\Aff_\theta(\Dcal)^\vee)\), respectively.

\begin{maincorollary}[{proved in Section~\ref{subsec:twining}}]
\label{cor:points}
Let \(K=\C\).
We write \(X_*(T)\) for the cocharacter lattice of \(T\), and let \(\rho^\vee\) be the sum of the fundamental coweights of \(G\).
For \(\lambda\in P_\ell(\Aff_\theta(\Dcal)^\vee)\), using the exponential map \(\exp:X_*(T)\otimes_\Z\C\to T(\C)\), set
\[
 t_\lambda=\exp\bigl(-2\pi\sqrt{-1}\,(\lambda+\rho^\vee)/d(\ell+h^\vee)\bigr)\in T(\C).
\]
\begin{enumerate}
\item There exists, up to isomorphism, exactly one object \((F,g)\) such that \(g\theta\) is \(G(\C)\)-conjugate to \(t_\lambda\theta\).
If we write \([t_\lambda\theta]\) for its isomorphism class, then the map
\[
 P_\ell(\Aff_\theta(\Dcal)^\vee)\longrightarrow
 [X(\beta^\ell\Delta)/T]^{\rot}(\C)/{\cong},\qquad
 \lambda\longmapsto[t_\lambda\theta]
\]
is bijective.
\item Under the equivalence of the main theorem, the basis element \(\chi_\mu\) for \(\mu\in P_\ell(\Aff_\theta(\Dcal^\vee)^\vee)\)
corresponds to the function \((F,g)\mapsto\operatorname{tr}(g\theta\mid V_\mu)\).
\item The value of this function at \([t_\lambda\theta]\) is \(S_{\mu\lambda}/S_{0\lambda}\).
\end{enumerate}
\end{maincorollary}

We outline the proofs of Theorem~\ref{thm:main} and Corollary~\ref{cor:points}.
A fixed point \((F,g)\) is determined by the conjugacy class of \(g\theta\), and by a result of He--Lusztig \cite{HeLusztig2012} (of Steinberg \cite{Steinberg1965} when \(\theta=\id\)), \(g\theta\) is represented by an element of the Steinberg cross-section.
Via the cross-section, the regular conjugacy classes correspond bijectively to the points of the quotient of the torus by \(W^\theta\), and the condition for the flag configuration to be closed becomes the condition defining a finite subgroup of the torus.
The set of its solutions is the spectrum of the fusion ring and is indexed by the lattice points in the fundamental alcove of the affine Weyl group.
The basis of the fusion ring corresponds to twining characters, and their values at each fixed point are ratios of entries of the \(S\)-matrix.
The following diagram shows how the proof proceeds.

\medskip
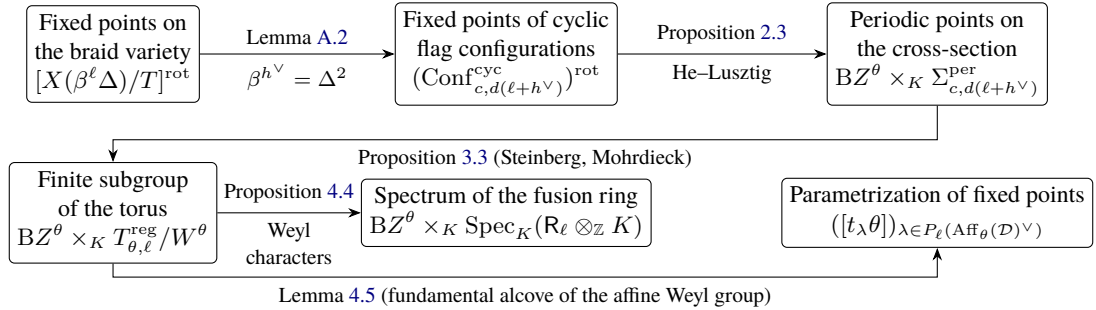
{\centering
\begin{tikzpicture}[x=1cm,y=1cm,>=Stealth,
 every node/.style={font=\footnotesize},
 stage/.style={draw,rounded corners=2pt,align=center,inner sep=3pt},
 above label/.style={above,font=\scriptsize},
 below label/.style={below,font=\scriptsize}]
\node[stage] (fix) at (0,0) {Fixed points on\\the braid variety\\\([X(\beta^\ell\Delta)/T]^{\rot}\)};
\node[stage] (conf) at (5.2,0) {Fixed points of cyclic\\flag configurations\\\((\Conf^{\mathrm{cyc}}_{c,d(\ell+h^\vee)})^{\rot}\)};
\node[stage] (slice) at (10.9,0) {Periodic points on\\the cross-section\\\(\BZ\times_K\Sigma_{c,d(\ell+h^\vee)}^{\mathrm{per}}\)};
\node[stage] (torus) at (0,-2.1) {Finite subgroup\\of the torus\\\(\BZ\times_K T_{\theta,\ell}^{\mathrm{reg}}/W^\theta\)};
\node[stage] (fusion) at (5.2,-2.1) {Spectrum of the fusion ring\\\(\BZ\times_K\Spec_K(\mathsf R_\ell\otimes_\Z K)\)};
\node[stage] (points) at (10.9,-2.1) {Parametrization of fixed points\\\(([t_\lambda\theta])_{\lambda\in P_\ell(\Aff_\theta(\Dcal)^\vee)}\)};
\draw[->] (fix) -- node[above label] {Lemma~\ref{lem:appendix-cut-rotation}} node[below label] {\(\beta^{h^\vee}=\Delta^2\)} (conf);
\draw[->] (conf) -- node[above label] {Proposition~\ref{prop:cut}} node[below label] {He--Lusztig} (slice);
\draw[->] (slice.south) -- ++(0,-0.45) -| node[pos=0.25,below label] {Proposition~\ref{prop:norm} (Steinberg, Mohrdieck)} (torus.north);
\draw[->] (torus) -- node[above label] {Proposition~\ref{prop:descent}} node[below label] {\shortstack{Weyl\\characters}} (fusion);
\draw[->] (torus.south) -- (0,-2.95) -- node[below label] {Lemma~\ref{lem:duality} (fundamental alcove of the affine Weyl group)} (10.9,-2.95) -- (points.south);
\end{tikzpicture}\par}
\medskip

Finally, we discuss the geometric background of the fixed-point stacks considered in this paper and directions for future research.
The motivation for considering fixed-point stacks comes from their relation to local systems on three-manifolds.
The fixed-point stack of the action induced by a surface homeomorphism on the moduli stack of local systems is equivalent to the moduli stack of local systems on its mapping torus.
Local systems on three-manifolds describe classical solutions of Chern--Simons theory, and the corresponding perturbative expansions appear in the study of quantum invariants and quantum modularity \cite{DimofteGukovLenellsZagier2009,GaroufalidisZagier2024}.
The cyclic flag configurations in this paper can be interpreted as local systems on a disk equipped with boundary flags satisfying prescribed relative position conditions \cite{FockGoncharov2006,GoncharovShen2019}.
Each rotation-fixed object gives rise, after passage to the adjoint group, to a framed local system on the lens space \(L\bigl(d(\ell+h^\vee),1\bigr)\).
Further study of the relation between this description and quantum invariants, and its extension to more general three-manifolds, are left for future work.

\paragraph*{Acknowledgments.}
The author had the idea for this paper while attending the
Workshop on Higher Teichmüller theory, knot theory and cluster algebras (July 13--17, 2026) held at University College Cork.
The author thanks the organizers Dani Kaufman and Robert Osburn, and the participants who had discussions with the author.
The author was supported by the Irish Research Council Advanced Laureate Award IRCLA/\allowbreak 2023/\allowbreak 1934 held by Robert Osburn.

\paragraph*{Use of AI.}
The author used ChatGPT--5.6 Sol in the research and in the preparation of this paper.
It was used for the refinement of proofs under the author's direction, checking of computations, literature search, and LaTeX editing.
The author takes responsibility for the mathematical results, proofs, citations, and the final text.

\section{Presentation of the fixed-point stack by the cross-section}
\label{sec:fixed}

We write \(U\) for the unipotent radical of \(B\).
Let \(\Phi\) be the root system of \((G,T)\) and \(\alpha_i\) the simple root corresponding to \(x_i\).
Since \(\theta\) preserves \(T\) and \(B\), it acts on \(\Phi\) and \(W\),
and satisfies \(\theta(\alpha_i)=\alpha_{\theta(i)}\) and \(\theta(s_i)=s_{\theta(i)}\).
Set \(I_\theta=I/\langle\theta\rangle\) and \(r=|I_\theta|\).
We also write \(\theta\) for the element \((1,\theta)\) of the semidirect product \(G\rtimes\langle\theta\rangle\), and
write \(G\theta\) for \(\{(g,\theta)\mid g\in G\}\).
Define the conjugation action of \(G\) on \(G\theta\) by
\(g\cdot\gamma=g\gamma g^{-1}\),
and call its orbits the \emph{\(G\)-conjugacy classes} of \(G\theta\).
We write the simple coroot and the fundamental weight corresponding to \(\alpha_i\)
as \(\alpha_i^\vee,\omega_i\), respectively,
and write \(P,Q\) for the weight lattice and the root lattice, respectively.
For a diagonalizable \(K\)-group \(D\), the group of \(K\)-group homomorphisms \(D\to\mathbb G_m\) is called
the \emph{character group} of \(D\) and is denoted by \(X^*(D)\).
We identify \(X^*(T)\) with \(P\).
In what follows, let \(c\) be a \(\theta\)-Coxeter element.

We write \(\Bcal\) for the flag variety \(G/B\).
Define the action of the semidirect product \(G\rtimes\langle\theta\rangle\) on \(\Bcal\)
by \((g\theta^a)\cdot hB=g\theta^a(h)B\) for \(a\in\Z/d\Z\).
The symbol \(BwB\) denotes the double coset \(BnB\)
given by a representative \(n\in N_G(T)\) of \(w\in W\).
Since \(T\subset B\), the double coset \(BnB\) does not depend on the choice of the representative.
Define the \emph{relative position} of flags \(F=gB\) and \(F'=g'B\) by
\[
 \pos(F,F')=w\quad\Longleftrightarrow\quad g^{-1}g'\in BwB.
\]

\begin{definition}
\label{def:fixed}
For a positive integer \(N\) divisible by \(d\), using cyclic indices \(k\in\Z/N\Z\), define the locally closed subscheme
\[
 C_{c,N}
 =\bigl\{F=(F_k)\in\Bcal^N\mid
   \forall\,k\in\Z/N\Z,\
   \pos(F_k,F_{k+1})=\theta^k(c)\bigr\}.
\]
We call the quotient stack by the diagonal \(G\)-action
\[
 \Conf^{\mathrm{cyc}}_{c,N}=[C_{c,N}/G]
\]
the stack of \emph{cyclic flag configurations}.
Define the \emph{rotation} \(\rot:C_{c,N}\to C_{c,N}\)
by \(\rot(F)_k=\theta\cdot F_{k-1}\).
Since \(\rot\) satisfies \(\rot(g\cdot F)=\theta(g)\cdot\rot(F)\) with respect to the diagonal action of \(G\),
it induces an autoequivalence of \(\Conf^{\mathrm{cyc}}_{c,N}\).
We also denote this autoequivalence by \(\rot\).
We define the \emph{fixed-point stack} \((\Conf^{\mathrm{cyc}}_{c,N})^{\rot}\) of \(\rot\) as follows.
For every \(K\)-scheme \(S\), let
\[
 (\Conf^{\mathrm{cyc}}_{c,N})^{\rot}(S)
 :=\bigl\{(\xi,\eta)\mid
   \xi\in\Conf^{\mathrm{cyc}}_{c,N}(S),\
   \eta:\rot\xi\xrightarrow{\sim}\xi\bigr\}.
\]
A morphism from \((\xi,\eta)\) to \((\xi',\eta')\) is
a morphism \(f:\xi\to\xi'\) satisfying \(f\circ\eta=\eta'\circ\rot(f)\).
Composition of morphisms and base change are those of \(\Conf^{\mathrm{cyc}}_{c,N}\).
\end{definition}

In Appendix~\ref{app:rotation}, we discuss, in the case \(N=d(\ell+h^\vee)\),
the presentation of cyclic flag configurations by the braid variety.

\begin{remark}
\label{rem:fixed-local}
For every \(K\)-scheme \(S\), the full subgroupoid consisting of those objects \((\xi,\eta)\) of the fixed-point stack over \(S\)
for which the \(G\)-torsor associated with \(\xi\) is trivial is equivalent to the groupoid
whose set of objects is
\[
 \bigl\{(F,g)\in C_{c,N}(S)\times G(S)\mid F=g\cdot\rot(F)\bigr\}
\]
and whose morphisms from \((F,g)\) to \((F',g')\) are the elements \(h\in G(S)\) satisfying
\(F'=h\cdot F\) and \(g'=hg\theta(h)^{-1}\).
General objects and morphisms are obtained from these by fppf descent.
In particular, if \(S\) is the spectrum of an algebraically closed field, then every \(G\)-torsor is trivial, so
the above groupoid is equivalent to the whole groupoid of the fixed-point stack over \(S\).
\end{remark}

Let \(U^-\) be the unipotent radical of the Borel subgroup containing \(T\) and opposite to \(B\).
Define \(\varphi_i:\mathrm{SL}_2\to G\) as the unique homomorphism satisfying the conditions
\[
 \varphi_i\begin{pmatrix}1&a\\0&1\end{pmatrix}=x_i(a),\qquad
 \varphi_i\begin{pmatrix}t&0\\0&t^{-1}\end{pmatrix}=\alpha_i^\vee(t).
\]
Set \[
 \dot s_i=\varphi_i\begin{pmatrix}0&-1\\1&0\end{pmatrix}.
\]
Since the \(\dot s_i\) satisfy the braid relations, for a reduced word \(i_1\cdots i_n\) of \(w\in W\), the product
\(\dot s_{i_1}\cdots\dot s_{i_n}\) along it does not depend on the choice of the reduced word. We denote it by \(\dot w\).
Set \[
 U^c=U\cap\dot cU^-\dot c^{-1},\qquad
 \Sigma_c=U^c\dot c\theta.
\]
We call \(\Sigma_c\) the \emph{twisted Steinberg cross-section}.
For a positive integer \(N\) divisible by \(d\), define the scheme-theoretic preimage of \(Z^\theta\) under the power map \(\Sigma_c\to G,\ \gamma\mapsto\gamma^N\)
by
\[
 \Sigma_{c,N}^{\mathrm{per}}
 =\{\gamma\in\Sigma_c\mid\gamma^N\in Z^\theta\}.
\]

From \(s_i\omega_j=\omega_j-\delta_{ij}\alpha_i\), we have \((1-c\theta)P=Q+(1-\theta)P\).
Hence the character groups of the fixed subgroup \(T^{c\theta}\) of the automorphism \(c\theta:T\to T\) and of \(Z^\theta\) are both \(P/(Q+(1-\theta)P)\), and
the natural inclusion \(Z^\theta\subset T^{c\theta}\) is an isomorphism.

We define the \emph{\(\theta\)-conjugacy class} of \(w\in W\) to be \(\{vw\theta(v)^{-1}\mid v\in W\}\).
When the fixed space of \(w\theta\) on \(P\otimes_\Z\Q\) is zero,
we call \(w\theta\) \emph{elliptic}.
The element \(c\theta\) is elliptic \cite[Lemma~7.4~(i)]{Springer1974}.
The element \(c\) also has minimal length in its \(\theta\)-conjugacy class.
Indeed, if a reduced word of an element \(w\) of this class contains no letter in a \(\theta\)-orbit \(o\),
then \(\sum_{i\in o}\omega_i\) is a nonzero vector fixed by \(w\theta\), contradicting ellipticity.
Hence every reduced word contains a letter in each orbit, and \(\len(w)\ge r=\len(c)\).

\begin{proposition}
\label{prop:cut}
For a positive integer \(N\) divisible by \(d\), there is a natural equivalence
\begin{equation}
 (\Conf^{\mathrm{cyc}}_{c,N})^{\rot}\simeq\BZ\times_K\Sigma_{c,N}^{\mathrm{per}}.
 \label{eq:cut-stack}
\end{equation}
In particular, the stabilizer group of each geometric point is \(Z^\theta\).
\end{proposition}

\begin{proof}
Using the description in Remark~\ref{rem:fixed-local},
we represent a fixed object fppf locally by a pair of a flag configuration \(F\) and an element \(g\in G\) satisfying \(F=g\cdot\rot(F)\).
Setting \(\gamma=g\theta\), we have \(F_k=\gamma\cdot F_{k-1}\), hence \(F_k=\gamma^k\cdot F_0\), and
\[
 \pos(F_0,\gamma\cdot F_0)=c,\qquad\gamma^N\cdot F_0=F_0.
\]
Conversely, when \(\gamma\in G\theta\) and a flag \(F_0\) satisfy these two conditions,
the flags \(F_k=\gamma^k\cdot F_0\) define a point of \(C_{c,N}\), since \(\gamma^k\) sends the relative position \(c\) to \(\theta^k(c)\).
The morphism \(h\) of Remark~\ref{rem:fixed-local} sends \((\gamma,F_0)\)
to \((h\gamma h^{-1},h\cdot F_0)\).
In an fppf local presentation with \(F_0=B\), we have \(\gamma\in B\dot cB\theta\),
and the remaining action is conjugation by \(B\). Therefore
\begin{equation}
 (\Conf^{\mathrm{cyc}}_{c,N})^{\rot}
 \simeq
 [\{a\in B\dot cB\theta\mid a^N\in B\}/B].
 \label{eq:incidence}
\end{equation}
Since \(c\) is an element of minimal length in an elliptic \(\theta\)-conjugacy class, \cite[Theorem~3.12]{HeLusztig2012} gives an isomorphism of schemes
\begin{equation}
 U\times\Sigma_c\xrightarrow{\sim}U\dot cU\theta,
 \qquad(u,\gamma)\longmapsto u\gamma u^{-1}.
 \label{eq:HL-section}
\end{equation}
We extend the isomorphism~\eqref{eq:HL-section} to the conjugation action of \(B\).
Since \(T^{c\theta}=Z^\theta\), the homomorphism
\(T\to T,\ t\mapsto t(c\theta)(t)^{-1}\) is an isogeny with kernel \(Z^\theta\).
Write \(a\in B\dot cB\theta\) uniquely as \(a=t_0y\) with \(t_0\in T\) and \(y\in U\dot cU\theta\),
and lift \(t_0\) fppf locally to \(t\in T\) along this isogeny.
Then \(t^{-1}at\in U\dot cU\theta\), and by the factorization above,
for this \(t\), the pair \((u,\gamma)\) satisfying \(a=tu\gamma(tu)^{-1}\) is uniquely determined.
The lift \(t\) is determined up to multiplication by an element of \(Z^\theta\), and since \(Z^\theta\) is central,
\(\gamma\) and the coset \(tuZ^\theta\in B/Z^\theta\) are determined by \(a\).
Since this construction is compatible with base change, by descent we obtain
a \(B\)-equivariant isomorphism of schemes, with respect to the left multiplication of \(B\) on \(B/Z^\theta\)
and the conjugation action on \(B\dot cB\theta\),
\begin{equation}
 (B/Z^\theta)\times\Sigma_c\xrightarrow{\sim}B\dot cB\theta,
 \qquad(bZ^\theta,\gamma)\longmapsto b\gamma b^{-1}.
 \label{eq:B-factorization}
\end{equation}

In particular, the centralizer in \(B\) of \(\gamma\in\Sigma_c\) is \(Z^\theta\).
If \(\gamma^N\in B\), then \(\gamma^N\) commutes with \(\gamma\), and hence \(\gamma^N\in Z^\theta\).
Therefore the inverse of \eqref{eq:B-factorization} maps \(\{a\in B\dot cB\theta\mid a^N\in B\}\)
to \((B/Z^\theta)\times\Sigma_{c,N}^{\mathrm{per}}\).
Since \(B\) acts transitively on \(B/Z^\theta\) with stabilizer \(Z^\theta\),
we obtain \eqref{eq:cut-stack} from \eqref{eq:incidence}.
\end{proof}

\section{The adjoint quotient and finite subgroups of the torus}
\label{sec:norm}

We write \(Z_G(\gamma)\) for the centralizer of \(\gamma\in G\theta\) in \(G\).
We call \(\gamma\in G\theta\) \emph{regular} if \(\dim Z_G(\gamma)=r\).
Define the \emph{affine quotient} for the conjugation action by
\(G\theta/\!/G=\Spec_K\mathcal O(G\theta)^G\), and write \(\pi:G\theta\to G\theta/\!/G\) for the quotient morphism.

\begin{lemma}
\label{lem:quotient-section}
The restriction
\[
 \pi|_{\Sigma_c}:\Sigma_c\hookrightarrow G\theta\xrightarrow{\ \pi\ }G\theta/\!/G
\]
is an isomorphism compatible with field extensions.
For every algebraically closed extension field \(\Omega\) of \(K\), the set \(\Sigma_c(\Omega)\) consists of regular elements,
each regular element of \(G\theta(\Omega)\) is \(G(\Omega)\)-conjugate to exactly one element of \(\Sigma_c(\Omega)\), and
at each point \(\gamma\) of \(\Sigma_c(\Omega)\) we have
\begin{equation}
 T_\gamma(G\theta)=T_\gamma(G\cdot\gamma)\oplus T_\gamma\Sigma_c.
 \label{eq:transverse}
\end{equation}
\end{lemma}

\begin{proof}
The variety \(\Sigma_c\) is defined over \(K\), and taking invariants commutes with extension of the base field.
Whether a morphism is an isomorphism can be checked after a faithfully flat field extension.
Therefore it suffices to treat the case where \(K\) is algebraically closed and \(\Omega=K\).

Choose a word \(i_1\cdots i_r\) representing \(c\). Since the word is reduced, the morphism
\[
 \mathbb A_K^r\xrightarrow{\sim}\Sigma_c,\qquad
 (a_1,\ldots,a_r)\longmapsto
 x_{i_1}(a_1)\dot s_{i_1}\cdots x_{i_r}(a_r)\dot s_{i_r}\theta
\]
is an isomorphism \cite[Lemmas~7.2--7.3, Theorem~7.1]{Steinberg1965}, \cite[Lemma~5.1]{Mohrdieck2000}.
When \(\theta=\id\), applying
\cite[Theorems~1.4, 6.1, 7.9]{Steinberg1965} to this product presentation shows
that \(\pi|_{\Sigma_c}\) is an isomorphism,
that \(\Sigma_c\) consists of regular elements, and that each regular element is \(G\)-conjugate to exactly one element of \(\Sigma_c\).
Also when \(\theta\ne\id\), applying
\cite[Definition~5.1, Theorems~5.1--5.2]{Mohrdieck2000} to the same product presentation gives the same conclusions.

At \(\gamma\in\Sigma_c(K)\), the differential \(d\pi\) is an isomorphism on \(T_\gamma\Sigma_c\) and hence surjective, and
the dimension of \(\ker d\pi\) is \(\dim G-r\).
By conjugation invariance, \(\ker d\pi\) contains \(T_\gamma(G\cdot\gamma)\), and
since \(\gamma\) is regular, the dimension of \(T_\gamma(G\cdot\gamma)\) is also \(\dim G-r\).
Hence the two coincide, and we obtain \eqref{eq:transverse}.
\end{proof}

Let \(\overline K\) be an algebraic closure of \(K\).
For \(x\in G\rtimes\langle\theta\rangle\), we write the differential at the identity of its conjugation action on \(G\) as
\(\Ad(x)\).
\begin{proposition}
\label{prop:etale}
For a positive integer \(N\) divisible by \(d\), the scheme \(\Sigma_{c,N}^{\mathrm{per}}\) is finite étale.
\end{proposition}

\begin{proof}
Let \(\gamma\in\Sigma_{c,N}^{\mathrm{per}}(\overline K)\).
Using the differentials of right multiplication by \(\gamma^{-1}\) and by \(\gamma^{-N}\), respectively,
we identify \(T_\gamma(G\theta)\) and \(T_{\gamma^N}G\) with \(\operatorname{Lie}(G)\).
The differential of the power map is
\[
 d(\gamma\mapsto\gamma^N)_\gamma=\sum_{j=0}^{N-1}\Ad(\gamma)^j.
\]
Since \(\Ad(\gamma)^N=1\) and the characteristic is zero, the kernel of this differential is
\(\operatorname{im}(\Ad(\gamma)-1)=T_\gamma(G\cdot\gamma)\).
Since the tangent space of \(Z^\theta\) is zero, \eqref{eq:transverse} gives
\(T_\gamma\Sigma_{c,N}^{\mathrm{per}}=0\).
Since \(\Sigma_{c,N}^{\mathrm{per}}\) is of finite type and its tangent spaces at all geometric points are zero, it is finite étale.
\end{proof}

Let \((1-\theta)T\) be the image of the homomorphism \(t\mapsto t\theta(t)^{-1}\), and
let \(T_\theta=T/(1-\theta)T\).
For an object \(X\) on which \(\theta\) acts, we write \(X^\theta\) for the part of \(X\) fixed by \(\theta\).
The action of \(W^\theta\) on \(T\) induces an action on \(T_\theta\).
We write the quotient by this action as
\(T_\theta/W^\theta=\Spec_K\mathcal O(T_\theta)^{W^\theta}\).
Changing the representative of \([t]\in T_\theta\) corresponds to \(T\)-conjugation of \(t\theta\).
Hence the regularity of \(t\theta\) depends only on \([t]\).
We call the open part consisting of the \([t]\) for which \(t\theta\) is regular the \emph{regular part} of \(T_\theta\).
The map \(t\mapsto\pi(t\theta)\) induces an isomorphism \(\iota:T_\theta/W^\theta\xrightarrow{\sim}G\theta/\!/G\)
\cite[Corollary~6.4]{Steinberg1965}, \cite[Theorem~2.1, (2.40)--(2.41)]{Mohrdieck2000}.
Since taking invariants commutes with extension of the base field, these results over algebraically closed fields descend to \(K\).

Define the \emph{norm maps} of the lattice \(P\) and the torus \(T\), respectively, by
\[
 \begin{aligned}
 \Nm_\theta:P&\longrightarrow P^\theta,&
 \mu&\longmapsto\sum_{a=0}^{d-1}\theta^a(\mu),\\
 \Nm_\theta:T&\longrightarrow T^\theta,&
 t&\longmapsto\prod_{a=0}^{d-1}\theta^a(t).
 \end{aligned}
\]
For \(\mu\in P\), we have
\(\mu(\Nm_\theta(t))=\Nm_\theta(\mu)(t)\).
The map \(\Nm_\theta\) on \(T\) is trivial on \((1-\theta)T\) and induces \(T_\theta\to T^\theta\).
We also write it as \(\Nm_\theta\).
For \(\ell\in\Z_{\ge0}\), define the closed subgroup scheme \(T_{\theta,\ell}\) of \(T_\theta\) by
\begin{equation}
 T_{\theta,\ell}
 =\bigl\{t\in T_\theta\bigm|\Nm_\theta(t)^{\ell+h^\vee}\in Z(G)\bigr\}.
 \label{eq:norm-torus}
\end{equation}
We write the regular part of \(T_{\theta,\ell}\) as
\(T_{\theta,\ell}^{\mathrm{reg}}\).

\begin{proposition}
\label{prop:norm}
The group \(T_{\theta,\ell}\) is a finite étale diagonalizable group, and
\begin{equation}
 X^*(T_{\theta,\ell})=P^\theta/(\ell+h^\vee)\Nm_\theta(Q),\qquad
 \Sigma_{c,d(\ell+h^\vee)}^{\mathrm{per}}\cong T_{\theta,\ell}^{\mathrm{reg}}/W^\theta.
 \label{eq:norm-quotient}
\end{equation}
The latter isomorphism is induced by \(\iota^{-1}\circ\pi\).
\end{proposition}

\begin{proof}
The group scheme \(T_{\theta,\ell}\) is the kernel of the morphism \(T_\theta\to T/Z(G)\), \(t\mapsto\Nm_\theta(t)^{\ell+h^\vee}Z(G)\).
We have \(X^*(T_\theta)=P^\theta\) and \(X^*(T/Z(G))=Q\), and the map on character groups induced by this morphism is
\(q\mapsto(\ell+h^\vee)\Nm_\theta(q)\).
Since its image is a sublattice of finite index in \(P^\theta\), we obtain the formula for the character group and the finite étaleness.

Set \(N=d(\ell+h^\vee)\).
For \(t\in T(\overline K)\), since \((t\theta)^N=\Nm_\theta(t)^{\ell+h^\vee}\in T^\theta\),
the condition \([t]\in T_{\theta,\ell}\) is equivalent to \((t\theta)^N\in Z^\theta\).
By Lemma~\ref{lem:quotient-section}, an element of \(G\theta(\overline K)\) is regular if and only if it is \(G(\overline K)\)-conjugate to an element of \(\Sigma_c(\overline K)\).
Since \(Z^\theta\) is central, the condition \(\gamma^N\in Z^\theta\) is preserved by \(G\)-conjugation.
Each \(\gamma\in\Sigma_{c,N}^{\mathrm{per}}(\overline K)\) has finite order since \(\gamma^N\in Z^\theta\), so it is a semisimple element of \(G\rtimes\langle\theta\rangle\) and
is conjugate to \(t\theta\) for some \(t\in T(\overline K)\) \cite[Proposition~2.3]{Mohrdieck2000}.
Hence \(\pi(\Sigma_{c,N}^{\mathrm{per}})\) and \(\iota(T_{\theta,\ell}^{\mathrm{reg}}/W^\theta)\) consist of the same geometric points of \(G\theta/\!/G\).
Since \(\pi|_{\Sigma_c}\) is an isomorphism, the former is a reduced closed subscheme by Proposition~\ref{prop:etale}.
In characteristic zero, the functor of invariants under a finite group is exact, so
the closed embedding \(T_{\theta,\ell}^{\mathrm{reg}}\hookrightarrow T_\theta\) induces a closed embedding of the quotients, and since \(\iota\) is an isomorphism, the latter is also a reduced closed subscheme.
Therefore the two coincide.
\end{proof}

\section{Fusion rings}
\label{sec:fusion}

In Section~\ref{subsec:fusion-def}, we define the fusion rings of general affine Dynkin diagrams.
In Section~\ref{subsec:fusion-comparison}, we identify \(T_{\theta,\ell}^{\mathrm{reg}}/W^\theta\) of Proposition~\ref{prop:norm} with the spectrum of \(\mathsf R_\ell(\Aff_\theta(\Dcal^\vee)^\vee)\) and prove the main theorem, Theorem~\ref{thm:main},
and in Section~\ref{subsec:twining}, we prove Corollary~\ref{cor:points}.

\subsection{Fusion rings of affine Lie algebras}
\label{subsec:fusion-def}

We write the vertex set of an affine Dynkin diagram \(\mathsf A\) as \(\{0\}\sqcup J\), using the vertex \(0\) of \cite[Tables~Aff~1--3]{Kac1990} and the set \(J\) of the remaining vertices.
We call the tuple \((a_j^\vee)_{j\in\{0\}\sqcup J}\) of relatively prime positive integers generating the left kernel of the Cartan matrix of \(\mathsf A\) (whose \((i,j)\) entry is the pairing of the simple coroot of the vertex \(i\) and the simple root of the vertex \(j\))
the dual Kac labels of \(\mathsf A\), and their sum \(h^\vee_{\mathsf A}\) the dual Coxeter number of \(\mathsf A\) \cite[\S6.1]{Kac1990}.
Then \(a_0^\vee=1\) \cite[(6.1.1)]{Kac1990}.
We call the Dynkin diagram obtained from \(\mathsf A\) by removing the vertex \(0\) the finite part of \(\mathsf A\).
We write \(\Phi_{\mathsf A}\), \(P_{\mathsf A}\), and \(W_{\mathsf A}\) for its root system, weight lattice, and Weyl group, respectively, and \(\omega_j\) for the fundamental weight of a vertex \(j\in J\).
The affine Weyl group of \(\mathsf A\) is the semidirect product \(W_{\mathsf A}\ltimes M_{\mathsf A}\) of \(W_{\mathsf A}\) and the lattice \(M_{\mathsf A}\subset P_{\mathsf A}\) of translations
\cite[\S6.5, Proposition~6.5]{Kac1990}.

\begin{definition}[{\cite[\S12.4]{Kac1990}}]
\label{def:level-weights}
For \(\ell\in\Z_{\ge0}\), define the finite subset \(P_\ell(\mathsf A)\) of \(P_{\mathsf A}\) by
\[
 P_\ell(\mathsf A)=\Bigl\{\sum_{j\in J}n_j\omega_j\Bigm|n_j\in\Z_{\ge0},\ \sum_{j\in J}a_j^\vee n_j\le\ell\Bigr\},
\]
and call its elements the \emph{dominant integral weights of level \(\ell\)} of \(\mathsf A\).
\end{definition}

We write \(T_{\mathsf A}\) for the split torus over \(\Q\) with character group \(P_{\mathsf A}\), and \(e^\mu\) for the function on \(T_{\mathsf A}\) corresponding to \(\mu\in P_{\mathsf A}\).
Set \[
 T_{\mathsf A,\ell}=\bigl\{t\in T_{\mathsf A}\bigm|\forall\mu\in(\ell+h^\vee_{\mathsf A})M_{\mathsf A},\ e^\mu(t)=1\bigr\},
\]
and write \(T_{\mathsf A,\ell}^{\mathrm{reg}}\) for the open and closed part of \(T_{\mathsf A,\ell}\) defined by \(e^\alpha\ne1\) for all \(\alpha\in\Phi_{\mathsf A}\) \cite[\S3.2]{Hong2018}.
The group \(W_{\mathsf A}\) acts on \(T_{\mathsf A}\) and preserves \(T_{\mathsf A,\ell}\) and \(T_{\mathsf A,\ell}^{\mathrm{reg}}\).
If we regard the Weyl characters \(\chi_\mu\in\Z[P_{\mathsf A}]^{W_{\mathsf A}}\) of \(\Phi_{\mathsf A}\) with highest weights \(\mu\in P_\ell(\mathsf A)\)
as functions on \(T_{\mathsf A,\ell}^{\mathrm{reg}}(\C)\), then they form a basis of the space of all functions on \(T_{\mathsf A,\ell}^{\mathrm{reg}}(\C)/W_{\mathsf A}\)
\cite[Theorem~2.1, Lemma~4.6]{Hong2018}, and their products are linear combinations of them with integer coefficients \cite[Theorem~2.2]{Hong2018}.

\begin{definition}[{\cite[\S2]{Hong2018}}]
\label{def:fusion}
We call the subring of the ring of functions on \(T_{\mathsf A,\ell}^{\mathrm{reg}}(\C)/W_{\mathsf A}\)
\[
 \mathsf R_\ell(\mathsf A)=\bigoplus_{\mu\in P_\ell(\mathsf A)}\Z\chi_\mu
\]
the \emph{fusion ring of level \(\ell\)} of \(\mathsf A\).
\end{definition}

\subsection{The quotient of the regular part and the fusion ring}
\label{subsec:fusion-comparison}

The vertex set of \(\Aff_\theta(\Dcal^\vee)^\vee\) is \(\{0\}\sqcup I_\theta\) \cite[\S6.1, \S8.3]{Kac1990}.
Using the realization by loop algebras \cite[Theorems~7.4, 8.3]{Kac1990}, we identify \(P_{\Aff_\theta(\Dcal^\vee)^\vee}\) with \(P^\theta\) via the map sending the fundamental weight \(\omega_o\) of a vertex \(o\in I_\theta\) to \(\sum_{i\in o}\omega_i\) \cite[\S7.4, \S8.3]{Kac1990}.
Under this identification, we write \(\Psi\subset P^\theta\) for the root system of the finite part of \(\Aff_\theta(\Dcal^\vee)^\vee\).
The group \(W^\theta\) acts as the Weyl group of \(\Psi\) \cite[Lemma~2.7, \S10]{Steinberg1959}.

We write \(e^\mu\) for the function on \(T\) corresponding to \(\mu\in P\).
When \(\mu\in P^\theta\), the function \(e^\mu\) defines a function on \(T_\theta\), and this identifies \(P^\theta\) with the character group of \(T_\theta\).
For \(t\in T\), the element \(t\theta\) is regular if and only if \(e^\psi(t)\ne1\) for all \(\psi\in\Psi\)
\cite[2.11]{Steinberg1965}, \cite[Lemmas~2.7--2.8, Corollary~3.3]{Mohrdieck2000}.

\begin{lemma}
\label{lem:walls}
We have \(M_{\Aff_\theta(\Dcal^\vee)^\vee}=\Nm_\theta(Q)\).
\end{lemma}

\begin{proof}
This follows from \cite[(6.5.8)]{Kac1990} when \(\theta=\id\), and from \cite[proof of Proposition~3.8]{Wendt2001} when \(\theta\ne\id\).
\end{proof}

Since \(P_{\Aff_\theta(\Dcal^\vee)^\vee}=P^\theta\), we have \(T_\theta=T_{\Aff_\theta(\Dcal^\vee)^\vee}\otimes_\Q K\), and
by Lemma~\ref{lem:walls} and the formula for the character group in Proposition~\ref{prop:norm}, we have \(T_{\theta,\ell}=T_{\Aff_\theta(\Dcal^\vee)^\vee,\ell}\otimes_\Q K\).
By the criterion for regularity of \(t\theta\), we have \(T_{\theta,\ell}^{\mathrm{reg}}=T_{\Aff_\theta(\Dcal^\vee)^\vee,\ell}^{\mathrm{reg}}\otimes_\Q K\).

\begin{proposition}
\label{prop:descent}
The map sending Weyl characters to functions on \(T_{\theta,\ell}^{\mathrm{reg}}\) gives a \(K\)-algebra isomorphism
\[
 \mathsf R_\ell(\Aff_\theta(\Dcal^\vee)^\vee)\otimes_\Z K\xrightarrow{\sim}\mathcal O(T_{\theta,\ell}^{\mathrm{reg}})^{W^\theta}.
\]
This isomorphism is compatible with field extensions.
\end{proposition}

\begin{proof}
Since \(\chi_\mu\in\Z[P^\theta]^{W^\theta}\), sending the basis element \(\chi_\mu\) to \(\chi_\mu\) on \(T_{\Aff_\theta(\Dcal^\vee)^\vee,\ell}^{\mathrm{reg}}\)
defines a \(\Q\)-linear map \(f\).
Since taking invariants of a finite group commutes with extension of the base field, the map in the proposition is \(f\otimes_\Q K\) by \(T_{\theta,\ell}^{\mathrm{reg}}=T_{\Aff_\theta(\Dcal^\vee)^\vee,\ell}^{\mathrm{reg}}\otimes_\Q K\),
and \(f\otimes_\Q\C\) is an algebra isomorphism by Definition~\ref{def:fusion}.
The preservation of multiplication and of the unit, and the vanishing of the kernel and the cokernel, descend to \(f\) by the faithful flatness of \(\C/\Q\).
\end{proof}

\begin{proof}[Proof of Theorem~\ref{thm:main}]
Applying Proposition~\ref{prop:norm} and Proposition~\ref{prop:descent} in turn, we obtain
\[
 \Sigma_{c,d(\ell+h^\vee)}^{\mathrm{per}}
 \cong T_{\theta,\ell}^{\mathrm{reg}}/W^\theta
 \cong\Spec_K\bigl(\mathsf R_\ell(\Aff_\theta(\Dcal^\vee)^\vee)\otimes_\Z K\bigr).
\]
Combining this with Proposition~\ref{prop:cut} and the equivalence compatible with rotation in Lemma~\ref{lem:appendix-cut-rotation}, we obtain the main theorem.
\end{proof}

\subsection{Parametrization of fixed points}
\label{subsec:twining}

In this section, let \(K=\C\).
We write \(\rho\) for the sum of the fundamental weights of \(G\), \(\rho^\vee\) for the sum of the fundamental coweights, \(P^\vee\subset X_*(T)\otimes\Q\) for the coweight lattice, and \(Q^\vee\) for the coroot lattice, and
define \(\Nm_\theta\) on \(P^\vee\) by the same formula as on \(P\).
For \(\lambda\in P_\ell(\Aff_\theta(\Dcal)^\vee)\), we write \(\tau_\lambda\) for the image in \(T_\theta(\C)\) of the element \(t_\lambda\) of Corollary~\ref{cor:points}.

\begin{lemma}
\label{lem:duality}
The map \(\lambda\mapsto[\tau_\lambda]\) is a bijection from \(P_\ell(\Aff_\theta(\Dcal)^\vee)\) to \((T_{\theta,\ell}^{\mathrm{reg}}/W^\theta)(\C)\).
\end{lemma}

\begin{proof}
Applying Section~\ref{subsec:fusion-comparison} and Lemma~\ref{lem:walls} to the dual root system,
we have \(P_{\Aff_\theta(\Dcal)^\vee}=(P^\vee)^\theta\) and \(M_{\Aff_\theta(\Dcal)^\vee}=\Nm_\theta(Q^\vee)\), and
the sum of the fundamental weights of the finite part is \(\rho^\vee\).
The dual Coxeter number of \(\Aff_\theta(\Dcal)^\vee\) is also equal to \(h^\vee\).
Indeed, if \(\theta\ne\id\), then \(\Dcal^\vee=\Dcal\), and
if \(\theta=\id\), then the dual Coxeter numbers of the two are equal to the Coxeter numbers of the untwisted affine diagrams of \(\Dcal\) and \(\Dcal^\vee\) \cite[\S6.1]{Kac1990}.

If we take as bases the sums of the fundamental weights and of the fundamental coweights over each orbit, then
\(\Nm_\theta(Q)\) and \(\Nm_\theta(Q^\vee)\) are \(d\) times the dual lattices of \((P^\vee)^\theta\) and \(P^\theta\), respectively.
Therefore, by the description of the character group in Proposition~\ref{prop:norm}, we obtain the \(W^\theta\)-equivariant isomorphism
\[
 (P^\vee)^\theta/(\ell+h^\vee)\Nm_\theta(Q^\vee)
 \xrightarrow{\sim}T_{\theta,\ell}(\C),\qquad
 \lambda\longmapsto\bigl[\exp(-2\pi\sqrt{-1}\lambda/d(\ell+h^\vee))\bigr].
\]
The reflection hyperplanes of \(W^\theta\ltimes(\ell+h^\vee)\Nm_\theta(Q^\vee)\) are
\(\langle\psi,\lambda\rangle=d(\ell+h^\vee)n\) for \(\psi\in\Psi\) and \(n\in\Z\) \cite[\S6.5--6.6]{Kac1990}.
By the criterion for regularity in the previous section and the description of the fundamental alcove of the affine Weyl group \cite[\S6.6]{Kac1990},
each orbit of regular points is uniquely represented by a lattice point in the interior of the fundamental alcove scaled by \(\ell+h^\vee\).
Since the sum of the dual Kac labels is \(h^\vee\) and \(a_0^\vee=1\), by Definition~\ref{def:level-weights}
these lattice points are \(\lambda+\rho^\vee\) for \(\lambda\in P_\ell(\Aff_\theta(\Dcal)^\vee)\), and they are mapped to \(\tau_\lambda\) by the above isomorphism.
\end{proof}

\begin{proof}[Proof of Corollary~\ref{cor:points}]
(1) By Proposition~\ref{prop:cut} and Proposition~\ref{prop:norm}, the coarse moduli map can be written on geometric points as
\[
 [(F,g)]\longmapsto\iota^{-1}(\pi(g\theta))\in
 (T_{\theta,\ell}^{\mathrm{reg}}/W^\theta)(\C).
\]
By Lemma~\ref{lem:duality}, for each \(\lambda\) there is, up to isomorphism, exactly one object \((F,g)\) whose image is \([\tau_\lambda]\).
For this object, the element \(g\theta\) is regular and \(\pi(g\theta)=\pi(t_\lambda\theta)\), so
by Lemma~\ref{lem:quotient-section}, \(g\theta\) is conjugate to \(t_\lambda\theta\).

(2) Extend \(V_\mu\) to a representation of \(G\rtimes\langle\theta\rangle\) so that \(\theta\) fixes a highest weight vector.
We call \(g\mapsto\operatorname{tr}(g\theta\mid V_\mu)\) the twining character \cite[\S4.2]{FuchsSchellekensSchweigert1996}.
By Jantzen's character formula \cite{Jantzen1973} (the ordinary character when \(\theta=\id\)),
we have \(\operatorname{tr}(t\theta\mid V_\mu)=\chi_\mu(t)\) for \(t\in T(\C)\).
By the isomorphism of Proposition~\ref{prop:descent} and the description of the regular conjugacy classes (Lemma~\ref{lem:quotient-section}),
the basis element \(\chi_\mu\) corresponds to \((F,g)\mapsto\operatorname{tr}(g\theta\mid V_\mu)\).

(3) For \(\mu\in P^\theta\), set \(A_\mu=\sum_{w\in W^\theta}\det(w)e^{w(\mu)}\).
Here \(\det(w)\) is the determinant of \(w\) on \(P^\theta\).
By the definition of \(t_\lambda\), the Weyl character formula, and the formula for the \(S\)-matrix \cite[Corollary~6.4, Theorems~13.8, 13.9]{Kac1990}, we obtain
\[
 \chi_\mu(\tau_\lambda)
 =\frac{A_{\mu+\rho}(\tau_\lambda)}{A_\rho(\tau_\lambda)}
 =\frac{S_{\mu\lambda}}{S_{0\lambda}}.
\]
\end{proof}

\section{Example}
\label{sec:example}

Let \(\Dcal=A_3\), \(\theta=(1\,3)\), and \(\ell=3\).
Then \(G=\mathrm{SL}_4(\C)\), \(c=12\), \(\beta=1232\), and
\(\Aff_\theta(\Dcal)^\vee=\Aff_\theta(\Dcal^\vee)^\vee=B_2^{(1)}\).
The dominant integral weights of level \(3\) are
\[
 P_3(\Aff_\theta(\Dcal)^\vee)
 =P_3(B_2^{(1)})
 =\Bigl\{a\omega_1+b\omega_2\Bigm|a,b\in\Z_{\ge0},\ a+b\le3\Bigr\},
\]
which consists of 10 elements.
We write \(\tau_\lambda\) for the image of the element
\(t_\lambda=\exp(-2\pi\sqrt{-1}(\lambda+\rho^\vee)/12)\) of Corollary~\ref{cor:points}
in \(T_\theta=T/(1-\theta)T\).
Setting \(r=2a+b+3\) and \(s=b+1\),
the values \(u=\chi_{\omega_1}(\tau_\lambda)\) and \(v=\chi_{\omega_2}(\tau_\lambda)\) are
\[
 u=1+2\cos\frac{\pi r}{6}+2\cos\frac{\pi s}{6},\qquad
 v=4\cos\frac{\pi r}{12}\cos\frac{\pi s}{12}.
\]
The corresponding elements of the cross-section and sequences of flags are given by
\[
 g_\lambda=x_1(u)\dot s_1x_2(v)\dot s_2,\qquad
 \gamma_\lambda=g_\lambda\theta=(\pi|_{\Sigma_c})^{-1}(\pi(t_\lambda\theta)),\qquad
 F_{\lambda,k}=\gamma_\lambda^kB.
\]
By Corollary~\ref{cor:points},
the map \(\lambda\mapsto[(F_\lambda,g_\lambda)]\) is a bijection onto the 10 isomorphism classes.

By the equivalence of Lemma~\ref{lem:appendix-cut-rotation} and
the left and right reflection maps of Appendix~\ref{subsec:appendix-cluster},
we map \(F_\lambda\) into the cell \(\Conf^{123212}_{232123}\).
On this cell, \(\rot^{-1}\) becomes
the time evolution of the level 3 restricted \(Y\)-system of \(B_2\), and
the mutation sets are
\[
 M_0=\{(1,2),(3,1),(2,1),(2,3),(2,5)\},\qquad
 M_1=\{(2,2),(2,4)\}.
\]
We mutate in the order \(M_0,M_1\), and
write the cluster Poisson coordinates \cite[Definitions~3.12--3.13]{ShenWeng2021} of the vertices immediately before each mutation
as \(X_1^{(1)},X_2^{(1)}\) for \((1,2),(3,1)\), respectively,
and as \(X_i^{(2)}\) for \((2,i)\).
Set \(z_i^{(a)}=(1+X_i^{(a)})^{-1}\).
The weights at which all seven \(z_i^{(a)}\) are defined are the six weights
\((a,b)=(0,0),(3,0),(0,1),(2,1),(0,2),(1,2)\).
By the fixed-point condition of the rotation, the \(z_i^{(a)}\) at these six weights
satisfy the following algebraic equations
\cite[Definition~1.7]{IIKKN2013B}:
\begin{equation}
\begin{gathered}
 \frac{(z_1^{(1)})^2}{z_2^{(1)}}
 =\frac{(1-z_1^{(1)})^2}{(1-z_1^{(2)})(1-z_2^{(2)})^2(1-z_3^{(2)})},\\
 \frac{(z_2^{(1)})^2}{z_1^{(1)}}
 =\frac{(1-z_2^{(1)})^2}{(1-z_3^{(2)})(1-z_4^{(2)})^2(1-z_5^{(2)})},\\
 \frac{(z_1^{(2)})^2}{z_2^{(2)}}=(1-z_1^{(2)})^2,\quad
 \frac{(z_2^{(2)})^2}{z_1^{(2)}z_3^{(2)}}
 =\frac{(1-z_2^{(2)})^2}{1-z_1^{(1)}},\quad
 \frac{(z_3^{(2)})^2}{z_2^{(2)}z_4^{(2)}}=(1-z_3^{(2)})^2,\\
 \frac{(z_4^{(2)})^2}{z_3^{(2)}z_5^{(2)}}
 =\frac{(1-z_4^{(2)})^2}{1-z_2^{(1)}},\quad
 \frac{(z_5^{(2)})^2}{z_4^{(2)}}=(1-z_5^{(2)})^2.
\end{gathered}
\label{eq:B2-nahm}
\end{equation}
At these six weights, the solutions can be written as
\begin{equation}
\begin{gathered}
 z_1^{(1)}=z_2^{(1)}=\frac{v^2-u}{u^2},\qquad
 z_1^{(2)}=z_5^{(2)}=\frac{u}{v^2},\\
 z_2^{(2)}=z_4^{(2)}=\frac{u^2}{(v^2-u)^2},\qquad
 z_3^{(2)}=\frac{u(u^2-v^2+u)}{v^2(v^2-2u)^2}.
\end{gathered}
\label{eq:B2-nahm-values}
\end{equation}
At the remaining four weights, \(v=0\) at \((0,3)\), \(v^2-u=0\) at \((1,0),(2,0)\), and
\(u=v=0\) at \((1,1)\); hence in each case some of the \(z_i^{(a)}\) are not defined by the above formulas.

\appendix
\section{Double Bott--Samelson cells and \texorpdfstring{\(Y\)}{Y}-systems}
\label{app:rotation}

The following presentation by cyclic flag configurations and the periodicity of the rotation
follow by almost the same argument that Shen--Weng use for \(\theta=\id\) \cite[\S2.3, \S5]{ShenWeng2021}.
For the reader's convenience, we give the construction and the main points of the proofs, including the case with a \(\theta\)-twist.

\subsection{Rotation}
\label{subsec:appendix-geometric-rotation}

For positive braids \(p,q\), we write \(\Conf^p_q\) for the
\emph{undecorated double Bott--Samelson cell} of \cite[Definition~2.20]{ShenWeng2021}, regarded as a stack.
We write \(\varnothing\) for the identity element of \(\operatorname{Br}^+\).
Writing \(X(p\Delta)\) for the braid variety associated with \(p\Delta\), we have the presentation \cite[\S3.7]{CasalsEtAl2025}
\begin{equation}
 \Conf^p_\varnothing\simeq[X(p\Delta)/T].
 \label{eq:braid-presentation}
\end{equation}
For this quotient, we use the action under which each \(t\in T\) sends every flag \(gB\) of a sequence of flags to \(t\cdot gB=tgB\).

For \(i\in I\), the reflection maps \(r^i:\Conf^{pi}_\varnothing\to\Conf^p_i\), \({}_ir:\Conf^p_i\to\Conf^{ip}_\varnothing\) are defined \cite[\S2.3]{ShenWeng2021}.
When the Demazure products of the words \(pi^*\), \(ip\) are both \(w_0\), we write \(R_i:X(pi^*)\xrightarrow{\sim}X(ip)\) for the rightward cyclic rotation \cite[Lemma~3.10]{CasalsEtAl2025}.
This map is defined over \(K\), is compatible with base change, and satisfies \(R_i(t\cdot F)=s_i(t)\cdot R_i(F)\).
By \(i\Delta=\Delta i^*\) and \eqref{eq:braid-presentation}, the map induced by \(R_i:X(pi\Delta)\xrightarrow{\sim}X(ip\Delta)\) on the quotient stacks corresponds to \({}_ir\circ r^i\).

Writing \(c=i_1\cdots i_r\), the automorphism defined in Section~\ref{sec:main} is \(\rot=R_{i_1}\circ\cdots\circ R_{i_r}\circ\theta\), and the above equivariance gives \(\rot(t\cdot F)=(c\theta)(t)\cdot\rot(F)\).
Therefore \(\rot\) induces an autoequivalence of \([X(\beta^\ell\Delta)/T]\), and via \eqref{eq:braid-presentation}, we use the same symbol on \(\Conf^{\beta^\ell}_\varnothing\) as well.
This construction is compatible with braid relations and does not depend on the choice of reduced words.
We write \(\len\) for the length function of \(W\).

\begin{lemma}
\label{lem:appendix-full-twist}
The number \(h^\vee\) is equal to the dual Coxeter number of \(\Aff_\theta(\Dcal^\vee)^\vee\), and the following equality holds in \(\operatorname{Br}^+\):
\begin{equation}
 \beta^{h^\vee}
 =\Delta^2.
 \label{eq:appendix-full-twist}
\end{equation}
\end{lemma}

\begin{proof}
Set \(m=\operatorname{ord}(c\theta)=dh^\vee\).
If \(\theta=\id\), then \(rm=|\Phi|\) \cite[\S4]{Coxeter1951},
and if \(\theta\ne\id\), then \(c\theta\) partitions \(\Phi\) into \(r\) orbits of length \(m\) \cite[\S7]{Springer1974}.
Hence \(rm=|\Phi|=2\len(w_0)\).
Moreover, comparing these orders \cite[Table~2]{Coxeter1951}, \cite[\S7]{Springer1974} with
the tables of dual Kac labels \cite[Tables~Aff~1--3]{Kac1990},
we obtain \(h^\vee=h^\vee_{\Aff_\theta(\Dcal^\vee)^\vee}\).

By Section~\ref{sec:fixed}, \(c\theta\) is elliptic, and the minimal length in its conjugacy class is \(r\).
Let \(\operatorname{Br}\) be the braid group containing \(\operatorname{Br}^+\).
By \cite[Proposition~5.3, \S5.4]{HeNie2012}, we can choose a conjugate \(v\theta\) of minimal length
such that \((v\theta)^m\in\Delta^2\operatorname{Br}^+\).
Since \(mr=2\len(w_0)\), comparing word lengths gives \((v\theta)^m=\Delta^2\).
Since \(v\theta\) and \(c\theta\) are conjugate in \(\operatorname{Br}\rtimes\langle\theta\rangle\) \cite[Theorem~3.2~(2)]{HeNie2012}
and \(\Delta^2\) belongs to its center, we obtain \((c\theta)^m=\Delta^2\).
The claim follows from \((c\theta)^d=\beta\) and the injectivity of the natural map from \(\operatorname{Br}^+\) to the braid group.
\end{proof}

\begin{lemma}
\label{lem:appendix-cut-rotation}
There is an equivalence of stacks
\begin{equation}
 \operatorname{cut}:
 \Conf^{\mathrm{cyc}}_{c,d(\ell+h^\vee)}
 \xrightarrow{\sim}\Conf^{\beta^\ell}_\varnothing
 \label{eq:appendix-cut}
\end{equation}
and with respect to this equivalence, we have
\(\rot\circ\operatorname{cut}\cong\operatorname{cut}\circ\rot\).
\end{lemma}

\begin{proof}
Let \(N=d(\ell+h^\vee)\), and let \(F=(F_k)_{k\in\Z/N\Z}\) be a cyclic flag configuration.
Moving the part \(F_{d\ell},\dots,F_N=F_0\) lying along \(\beta^{h^\vee}\)
by the braid relation \eqref{eq:appendix-full-twist},
we obtain the point of \(\Conf^{\beta^\ell}_\varnothing\)
\[
 \begin{tikzpicture}[baseline=(current bounding box.center),>=Stealth]
  \node (a) at (0,0) {\(F_0\)};
  \node (b) at (3,0) {\(F_{d\ell}\)};
  \node (c) at (1.5,-1.3) {\(F'\)};
  \draw[->] (a) -- node[above] {\(\beta^\ell\)} (b);
  \draw[->] (b) -- node[below right] {\(w_0\)} (c);
  \draw[->] (c) -- node[below left] {\(w_0\)} (a);
 \end{tikzpicture}.
\]
Let \(\operatorname{cut}(F)\) be this point.
Since an edge of relative position \(w_0\) can be uniquely restored to a sequence of flags along a reduced word,
this gives an equivalence \(\operatorname{cut}\) of quotient stacks.

The word going once around the triangle of \(\operatorname{cut}(F)\) is \(\beta^\ell\Delta^2\), and we have
\[
 \beta^\ell\Delta^2
 \overset{\theta}{\longmapsto}\theta(\beta)^\ell\Delta^2
 =\theta(\beta)^\ell ac^*\Delta
 =\theta(\beta)^\ell a\Delta c.
\]
By the uniqueness of sequences of flags along reduced words, \(\rot\) on \(\operatorname{cut}(F)\) corresponds to \((F_k)_k\mapsto(\theta\cdot F_{k-1})_k\) on \(F\).
Since each operation is defined over \(K\) and compatible with base change,
we obtain \(\rot\circ\operatorname{cut}\cong\operatorname{cut}\circ\rot\).
\end{proof}

\begin{corollary}
\label{cor:appendix-periodicity}
On \([X(\beta^\ell\Delta)/T]\), we have \(\rot^{d(\ell+h^\vee)}\cong\id\).
\end{corollary}

\begin{proof}
On \(\Conf^{\mathrm{cyc}}_{c,d(\ell+h^\vee)}\), we have \(\rot^{d(\ell+h^\vee)}\cong\id\) by definition, and
the claim follows from Lemma~\ref{lem:appendix-cut-rotation}.
\end{proof}

\subsection{Cluster coordinates and level restricted \texorpdfstring{\(Y\)}{Y}-systems}
\label{subsec:appendix-cluster}

In this subsection, we assume \(\ell\ge2\).
Shen--Weng \cite[Lemma~5.9, Corollary~5.10]{ShenWeng2021}
expressed the rotation as a sequence of mutations via flips of triangulations,
and proved the periodicity of the level \(\ell\) restricted \(Y\)-systems of type \(ADE\).
For general \((\Dcal,\theta)\) as well, the same construction expresses the rotation as a sequence of mutations.
In particular, Corollary~\ref{cor:appendix-periodicity} also gives the periodicity of the level restricted \(Y\)-systems of type \(BCFG\).

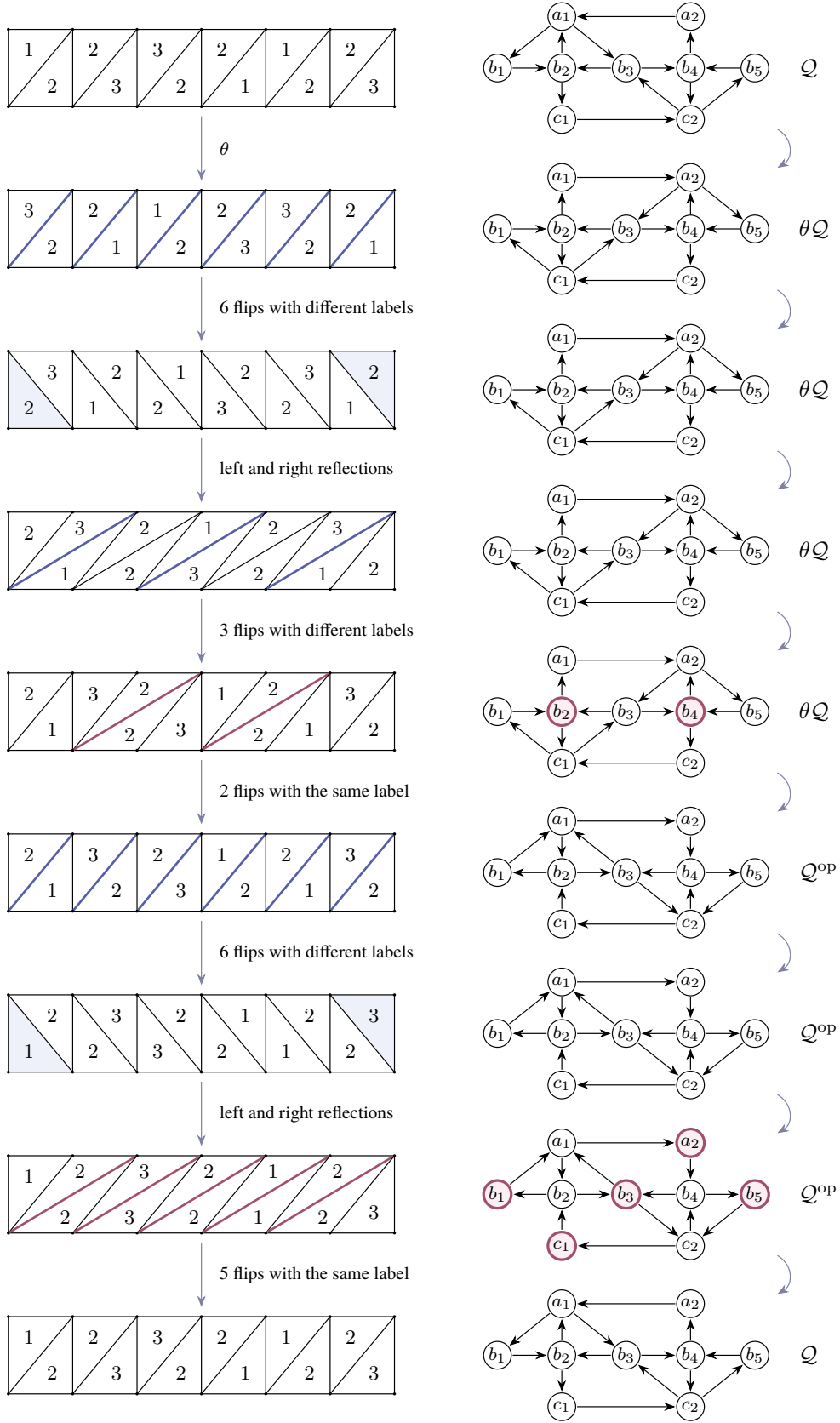
\begin{figure}[p]
\centering
\definecolor{cgfsame}{HTML}{A24F6C}
\definecolor{cgfsamefill}{HTML}{FBEFF3}
\definecolor{cgfdiff}{HTML}{565EA8}
\definecolor{cgfcorner}{HTML}{EEF0FA}
\definecolor{cgfstep}{HTML}{7C80A3}
\begin{tikzpicture}[x=1cm,y=1cm,>=Stealth,line width=.45pt,
 every node/.style={font=\small},
 qvertex/.style={circle,draw,fill=white,inner sep=.4pt,minimum size=4.3mm,font=\scriptsize},
 qarrow/.style={->,shorten >=.4pt,shorten <=.4pt},
 transition/.style={->,cgfstep},
 samediag/.style={cgfsame,line width=1pt},
 diffdiag/.style={cgfdiff,line width=1pt},
 cornertri/.style={fill=cgfcorner},
 mutvertex/.style={fill=cgfsamefill,draw=cgfsame,line width=1.2pt}]

\begin{scope}[yshift=-0.000cm]
\draw (0,0.6)--(6,0.6)--(6,-0.6)--(0,-0.6)--cycle;
\draw (1,0.6)--(0,-0.6);
\draw (1,0.6)--(1,-0.6);
\draw (2,0.6)--(1,-0.6);
\draw (2,0.6)--(2,-0.6);
\draw (3,0.6)--(2,-0.6);
\draw (3,0.6)--(3,-0.6);
\draw (4,0.6)--(3,-0.6);
\draw (4,0.6)--(4,-0.6);
\draw (5,0.6)--(4,-0.6);
\draw (5,0.6)--(5,-0.6);
\draw (6,0.6)--(5,-0.6);
\node[inner sep=0pt,font=\footnotesize] at (0.319,0.281) {$ 1 $};
\node[inner sep=0pt,font=\footnotesize] at (0.681,-0.281) {$ 2 $};
\node[inner sep=0pt,font=\footnotesize] at (1.319,0.281) {$ 2 $};
\node[inner sep=0pt,font=\footnotesize] at (1.681,-0.281) {$ 3 $};
\node[inner sep=0pt,font=\footnotesize] at (2.319,0.281) {$ 3 $};
\node[inner sep=0pt,font=\footnotesize] at (2.681,-0.281) {$ 2 $};
\node[inner sep=0pt,font=\footnotesize] at (3.319,0.281) {$ 2 $};
\node[inner sep=0pt,font=\footnotesize] at (3.681,-0.281) {$ 1 $};
\node[inner sep=0pt,font=\footnotesize] at (4.319,0.281) {$ 1 $};
\node[inner sep=0pt,font=\footnotesize] at (4.681,-0.281) {$ 2 $};
\node[inner sep=0pt,font=\footnotesize] at (5.319,0.281) {$ 2 $};
\node[inner sep=0pt,font=\footnotesize] at (5.681,-0.281) {$ 3 $};
\fill (0,0.6) circle (.65pt); \fill (0,-0.6) circle (.65pt);
\fill (1,0.6) circle (.65pt); \fill (1,-0.6) circle (.65pt);
\fill (2,0.6) circle (.65pt); \fill (2,-0.6) circle (.65pt);
\fill (3,0.6) circle (.65pt); \fill (3,-0.6) circle (.65pt);
\fill (4,0.6) circle (.65pt); \fill (4,-0.6) circle (.65pt);
\fill (5,0.6) circle (.65pt); \fill (5,-0.6) circle (.65pt);
\fill (6,0.6) circle (.65pt); \fill (6,-0.6) circle (.65pt);
\node[anchor=west] at (12.150,0) {$ \mathcal Q $};
\node[qvertex] (r0v0) at (8.600,0.800) {$ a_{1} $};
\node[qvertex] (r0v1) at (10.600,0.800) {$ a_{2} $};
\node[qvertex] (r0v2) at (8.600,-0.800) {$ c_{1} $};
\node[qvertex] (r0v3) at (10.600,-0.800) {$ c_{2} $};
\node[qvertex] (r0v4) at (7.600,0.000) {$ b_{1} $};
\node[qvertex] (r0v5) at (8.600,0.000) {$ b_{2} $};
\node[qvertex] (r0v6) at (9.600,0.000) {$ b_{3} $};
\node[qvertex] (r0v7) at (10.600,0.000) {$ b_{4} $};
\node[qvertex] (r0v8) at (11.600,0.000) {$ b_{5} $};
\draw[qarrow] (r0v0) to (r0v4);
\draw[qarrow] (r0v0) to (r0v6);
\draw[qarrow] (r0v1) to (r0v0);
\draw[qarrow] (r0v2) to (r0v3);
\draw[qarrow] (r0v3) to (r0v6);
\draw[qarrow] (r0v3) to (r0v8);
\draw[qarrow] (r0v4) to (r0v5);
\draw[qarrow] (r0v5) to (r0v0);
\draw[qarrow] (r0v5) to (r0v2);
\draw[qarrow] (r0v6) to (r0v5);
\draw[qarrow] (r0v6) to (r0v7);
\draw[qarrow] (r0v7) to (r0v1);
\draw[qarrow] (r0v7) to (r0v3);
\draw[qarrow] (r0v8) to (r0v7);
\draw[transition] (3,-0.750)--(3,-1.750);
\node[anchor=west,font=\scriptsize] at (3.15,-1.250) {$\theta$};
\draw[transition] (11.950,-0.950) .. controls (12.200,-1.100) and (12.200,-1.400) .. (11.950,-1.550);
\end{scope}
\begin{scope}[yshift=-2.500cm]
\draw (0,0.6)--(6,0.6)--(6,-0.6)--(0,-0.6)--cycle;
\draw[diffdiag] (1,0.6)--(0,-0.6);
\draw (1,0.6)--(1,-0.6);
\draw[diffdiag] (2,0.6)--(1,-0.6);
\draw (2,0.6)--(2,-0.6);
\draw[diffdiag] (3,0.6)--(2,-0.6);
\draw (3,0.6)--(3,-0.6);
\draw[diffdiag] (4,0.6)--(3,-0.6);
\draw (4,0.6)--(4,-0.6);
\draw[diffdiag] (5,0.6)--(4,-0.6);
\draw (5,0.6)--(5,-0.6);
\draw[diffdiag] (6,0.6)--(5,-0.6);
\node[inner sep=0pt,font=\footnotesize] at (0.319,0.281) {$ 3 $};
\node[inner sep=0pt,font=\footnotesize] at (0.681,-0.281) {$ 2 $};
\node[inner sep=0pt,font=\footnotesize] at (1.319,0.281) {$ 2 $};
\node[inner sep=0pt,font=\footnotesize] at (1.681,-0.281) {$ 1 $};
\node[inner sep=0pt,font=\footnotesize] at (2.319,0.281) {$ 1 $};
\node[inner sep=0pt,font=\footnotesize] at (2.681,-0.281) {$ 2 $};
\node[inner sep=0pt,font=\footnotesize] at (3.319,0.281) {$ 2 $};
\node[inner sep=0pt,font=\footnotesize] at (3.681,-0.281) {$ 3 $};
\node[inner sep=0pt,font=\footnotesize] at (4.319,0.281) {$ 3 $};
\node[inner sep=0pt,font=\footnotesize] at (4.681,-0.281) {$ 2 $};
\node[inner sep=0pt,font=\footnotesize] at (5.319,0.281) {$ 2 $};
\node[inner sep=0pt,font=\footnotesize] at (5.681,-0.281) {$ 1 $};
\fill (0,0.6) circle (.65pt); \fill (0,-0.6) circle (.65pt);
\fill (1,0.6) circle (.65pt); \fill (1,-0.6) circle (.65pt);
\fill (2,0.6) circle (.65pt); \fill (2,-0.6) circle (.65pt);
\fill (3,0.6) circle (.65pt); \fill (3,-0.6) circle (.65pt);
\fill (4,0.6) circle (.65pt); \fill (4,-0.6) circle (.65pt);
\fill (5,0.6) circle (.65pt); \fill (5,-0.6) circle (.65pt);
\fill (6,0.6) circle (.65pt); \fill (6,-0.6) circle (.65pt);
\node[anchor=west] at (12.150,0) {$ \theta\mathcal Q $};
\node[qvertex] (r1v0) at (8.600,0.800) {$ a_{1} $};
\node[qvertex] (r1v1) at (10.600,0.800) {$ a_{2} $};
\node[qvertex] (r1v2) at (8.600,-0.800) {$ c_{1} $};
\node[qvertex] (r1v3) at (10.600,-0.800) {$ c_{2} $};
\node[qvertex] (r1v4) at (7.600,0.000) {$ b_{1} $};
\node[qvertex] (r1v5) at (8.600,0.000) {$ b_{2} $};
\node[qvertex] (r1v6) at (9.600,0.000) {$ b_{3} $};
\node[qvertex] (r1v7) at (10.600,0.000) {$ b_{4} $};
\node[qvertex] (r1v8) at (11.600,0.000) {$ b_{5} $};
\draw[qarrow] (r1v0) to (r1v1);
\draw[qarrow] (r1v1) to (r1v6);
\draw[qarrow] (r1v1) to (r1v8);
\draw[qarrow] (r1v2) to (r1v4);
\draw[qarrow] (r1v2) to (r1v6);
\draw[qarrow] (r1v3) to (r1v2);
\draw[qarrow] (r1v4) to (r1v5);
\draw[qarrow] (r1v5) to (r1v0);
\draw[qarrow] (r1v5) to (r1v2);
\draw[qarrow] (r1v6) to (r1v5);
\draw[qarrow] (r1v6) to (r1v7);
\draw[qarrow] (r1v7) to (r1v1);
\draw[qarrow] (r1v7) to (r1v3);
\draw[qarrow] (r1v8) to (r1v7);
\draw[transition] (3,-0.750)--(3,-1.750);
\node[anchor=west,font=\scriptsize] at (3.15,-1.250) {6 flips with different labels};
\draw[transition] (11.950,-0.950) .. controls (12.200,-1.100) and (12.200,-1.400) .. (11.950,-1.550);
\end{scope}
\begin{scope}[yshift=-5.000cm]
\fill[cornertri] (0,-0.6)--(1,-0.6)--(0,0.6);
\fill[cornertri] (5,0.6)--(6,0.6)--(6,-0.6);
\draw (0,0.6)--(6,0.6)--(6,-0.6)--(0,-0.6)--cycle;
\draw (0,0.6)--(1,-0.6);
\draw (1,0.6)--(1,-0.6);
\draw (1,0.6)--(2,-0.6);
\draw (2,0.6)--(2,-0.6);
\draw (2,0.6)--(3,-0.6);
\draw (3,0.6)--(3,-0.6);
\draw (3,0.6)--(4,-0.6);
\draw (4,0.6)--(4,-0.6);
\draw (4,0.6)--(5,-0.6);
\draw (5,0.6)--(5,-0.6);
\draw (5,0.6)--(6,-0.6);
\node[inner sep=0pt,font=\footnotesize] at (0.319,-0.281) {$ 2 $};
\node[inner sep=0pt,font=\footnotesize] at (0.681,0.281) {$ 3 $};
\node[inner sep=0pt,font=\footnotesize] at (1.319,-0.281) {$ 1 $};
\node[inner sep=0pt,font=\footnotesize] at (1.681,0.281) {$ 2 $};
\node[inner sep=0pt,font=\footnotesize] at (2.319,-0.281) {$ 2 $};
\node[inner sep=0pt,font=\footnotesize] at (2.681,0.281) {$ 1 $};
\node[inner sep=0pt,font=\footnotesize] at (3.319,-0.281) {$ 3 $};
\node[inner sep=0pt,font=\footnotesize] at (3.681,0.281) {$ 2 $};
\node[inner sep=0pt,font=\footnotesize] at (4.319,-0.281) {$ 2 $};
\node[inner sep=0pt,font=\footnotesize] at (4.681,0.281) {$ 3 $};
\node[inner sep=0pt,font=\footnotesize] at (5.319,-0.281) {$ 1 $};
\node[inner sep=0pt,font=\footnotesize] at (5.681,0.281) {$ 2 $};
\fill (0,0.6) circle (.65pt); \fill (0,-0.6) circle (.65pt);
\fill (1,0.6) circle (.65pt); \fill (1,-0.6) circle (.65pt);
\fill (2,0.6) circle (.65pt); \fill (2,-0.6) circle (.65pt);
\fill (3,0.6) circle (.65pt); \fill (3,-0.6) circle (.65pt);
\fill (4,0.6) circle (.65pt); \fill (4,-0.6) circle (.65pt);
\fill (5,0.6) circle (.65pt); \fill (5,-0.6) circle (.65pt);
\fill (6,0.6) circle (.65pt); \fill (6,-0.6) circle (.65pt);
\node[anchor=west] at (12.150,0) {$ \theta\mathcal Q $};
\node[qvertex] (r2v0) at (8.600,0.800) {$ a_{1} $};
\node[qvertex] (r2v1) at (10.600,0.800) {$ a_{2} $};
\node[qvertex] (r2v2) at (8.600,-0.800) {$ c_{1} $};
\node[qvertex] (r2v3) at (10.600,-0.800) {$ c_{2} $};
\node[qvertex] (r2v4) at (7.600,0.000) {$ b_{1} $};
\node[qvertex] (r2v5) at (8.600,0.000) {$ b_{2} $};
\node[qvertex] (r2v6) at (9.600,0.000) {$ b_{3} $};
\node[qvertex] (r2v7) at (10.600,0.000) {$ b_{4} $};
\node[qvertex] (r2v8) at (11.600,0.000) {$ b_{5} $};
\draw[qarrow] (r2v0) to (r2v1);
\draw[qarrow] (r2v1) to (r2v6);
\draw[qarrow] (r2v1) to (r2v8);
\draw[qarrow] (r2v2) to (r2v4);
\draw[qarrow] (r2v2) to (r2v6);
\draw[qarrow] (r2v3) to (r2v2);
\draw[qarrow] (r2v4) to (r2v5);
\draw[qarrow] (r2v5) to (r2v0);
\draw[qarrow] (r2v5) to (r2v2);
\draw[qarrow] (r2v6) to (r2v5);
\draw[qarrow] (r2v6) to (r2v7);
\draw[qarrow] (r2v7) to (r2v1);
\draw[qarrow] (r2v7) to (r2v3);
\draw[qarrow] (r2v8) to (r2v7);
\draw[transition] (3,-0.750)--(3,-1.750);
\node[anchor=west,font=\scriptsize] at (3.15,-1.250) {left and right reflections};
\draw[transition] (11.950,-0.950) .. controls (12.200,-1.100) and (12.200,-1.400) .. (11.950,-1.550);
\end{scope}
\begin{scope}[yshift=-7.500cm]
\draw (0,0.6)--(6,0.6)--(6,-0.6)--(0,-0.6)--cycle;
\draw (1,0.6)--(0,-0.6);
\draw[diffdiag] (2,0.6)--(0,-0.6);
\draw (2,0.6)--(1,-0.6);
\draw (3,0.6)--(1,-0.6);
\draw (3,0.6)--(2,-0.6);
\draw[diffdiag] (4,0.6)--(2,-0.6);
\draw (4,0.6)--(3,-0.6);
\draw (5,0.6)--(3,-0.6);
\draw (5,0.6)--(4,-0.6);
\draw[diffdiag] (6,0.6)--(4,-0.6);
\draw (6,0.6)--(5,-0.6);
\node[inner sep=0pt,font=\footnotesize] at (0.319,0.281) {$ 2 $};
\node[inner sep=0pt,font=\footnotesize] at (1.115,0.355) {$ 3 $};
\node[inner sep=0pt,font=\footnotesize] at (0.885,-0.355) {$ 1 $};
\node[inner sep=0pt,font=\footnotesize] at (2.115,0.355) {$ 2 $};
\node[inner sep=0pt,font=\footnotesize] at (1.885,-0.355) {$ 2 $};
\node[inner sep=0pt,font=\footnotesize] at (3.115,0.355) {$ 1 $};
\node[inner sep=0pt,font=\footnotesize] at (2.885,-0.355) {$ 3 $};
\node[inner sep=0pt,font=\footnotesize] at (4.115,0.355) {$ 2 $};
\node[inner sep=0pt,font=\footnotesize] at (3.885,-0.355) {$ 2 $};
\node[inner sep=0pt,font=\footnotesize] at (5.115,0.355) {$ 3 $};
\node[inner sep=0pt,font=\footnotesize] at (4.885,-0.355) {$ 1 $};
\node[inner sep=0pt,font=\footnotesize] at (5.681,-0.281) {$ 2 $};
\fill (0,0.6) circle (.65pt); \fill (0,-0.6) circle (.65pt);
\fill (1,0.6) circle (.65pt); \fill (1,-0.6) circle (.65pt);
\fill (2,0.6) circle (.65pt); \fill (2,-0.6) circle (.65pt);
\fill (3,0.6) circle (.65pt); \fill (3,-0.6) circle (.65pt);
\fill (4,0.6) circle (.65pt); \fill (4,-0.6) circle (.65pt);
\fill (5,0.6) circle (.65pt); \fill (5,-0.6) circle (.65pt);
\fill (6,0.6) circle (.65pt); \fill (6,-0.6) circle (.65pt);
\node[anchor=west] at (12.150,0) {$ \theta\mathcal Q $};
\node[qvertex] (r3v0) at (8.600,0.800) {$ a_{1} $};
\node[qvertex] (r3v1) at (10.600,0.800) {$ a_{2} $};
\node[qvertex] (r3v2) at (8.600,-0.800) {$ c_{1} $};
\node[qvertex] (r3v3) at (10.600,-0.800) {$ c_{2} $};
\node[qvertex] (r3v4) at (7.600,0.000) {$ b_{1} $};
\node[qvertex] (r3v5) at (8.600,0.000) {$ b_{2} $};
\node[qvertex] (r3v6) at (9.600,0.000) {$ b_{3} $};
\node[qvertex] (r3v7) at (10.600,0.000) {$ b_{4} $};
\node[qvertex] (r3v8) at (11.600,0.000) {$ b_{5} $};
\draw[qarrow] (r3v0) to (r3v1);
\draw[qarrow] (r3v1) to (r3v6);
\draw[qarrow] (r3v1) to (r3v8);
\draw[qarrow] (r3v2) to (r3v4);
\draw[qarrow] (r3v2) to (r3v6);
\draw[qarrow] (r3v3) to (r3v2);
\draw[qarrow] (r3v4) to (r3v5);
\draw[qarrow] (r3v5) to (r3v0);
\draw[qarrow] (r3v5) to (r3v2);
\draw[qarrow] (r3v6) to (r3v5);
\draw[qarrow] (r3v6) to (r3v7);
\draw[qarrow] (r3v7) to (r3v1);
\draw[qarrow] (r3v7) to (r3v3);
\draw[qarrow] (r3v8) to (r3v7);
\draw[transition] (3,-0.750)--(3,-1.750);
\node[anchor=west,font=\scriptsize] at (3.15,-1.250) {3 flips with different labels};
\draw[transition] (11.950,-0.950) .. controls (12.200,-1.100) and (12.200,-1.400) .. (11.950,-1.550);
\end{scope}
\begin{scope}[yshift=-10.000cm]
\draw (0,0.6)--(6,0.6)--(6,-0.6)--(0,-0.6)--cycle;
\draw (1,0.6)--(0,-0.6);
\draw (1,0.6)--(1,-0.6);
\draw (2,0.6)--(1,-0.6);
\draw[samediag] (3,0.6)--(1,-0.6);
\draw (3,0.6)--(2,-0.6);
\draw (3,0.6)--(3,-0.6);
\draw (4,0.6)--(3,-0.6);
\draw[samediag] (5,0.6)--(3,-0.6);
\draw (5,0.6)--(4,-0.6);
\draw (5,0.6)--(5,-0.6);
\draw (6,0.6)--(5,-0.6);
\node[inner sep=0pt,font=\footnotesize] at (0.319,0.281) {$ 2 $};
\node[inner sep=0pt,font=\footnotesize] at (0.681,-0.281) {$ 1 $};
\node[inner sep=0pt,font=\footnotesize] at (1.319,0.281) {$ 3 $};
\node[inner sep=0pt,font=\footnotesize] at (2.115,0.355) {$ 2 $};
\node[inner sep=0pt,font=\footnotesize] at (1.885,-0.355) {$ 2 $};
\node[inner sep=0pt,font=\footnotesize] at (2.681,-0.281) {$ 3 $};
\node[inner sep=0pt,font=\footnotesize] at (3.319,0.281) {$ 1 $};
\node[inner sep=0pt,font=\footnotesize] at (4.115,0.355) {$ 2 $};
\node[inner sep=0pt,font=\footnotesize] at (3.885,-0.355) {$ 2 $};
\node[inner sep=0pt,font=\footnotesize] at (4.681,-0.281) {$ 1 $};
\node[inner sep=0pt,font=\footnotesize] at (5.319,0.281) {$ 3 $};
\node[inner sep=0pt,font=\footnotesize] at (5.681,-0.281) {$ 2 $};
\fill (0,0.6) circle (.65pt); \fill (0,-0.6) circle (.65pt);
\fill (1,0.6) circle (.65pt); \fill (1,-0.6) circle (.65pt);
\fill (2,0.6) circle (.65pt); \fill (2,-0.6) circle (.65pt);
\fill (3,0.6) circle (.65pt); \fill (3,-0.6) circle (.65pt);
\fill (4,0.6) circle (.65pt); \fill (4,-0.6) circle (.65pt);
\fill (5,0.6) circle (.65pt); \fill (5,-0.6) circle (.65pt);
\fill (6,0.6) circle (.65pt); \fill (6,-0.6) circle (.65pt);
\node[anchor=west] at (12.150,0) {$ \theta\mathcal Q $};
\node[qvertex] (r4v0) at (8.600,0.800) {$ a_{1} $};
\node[qvertex] (r4v1) at (10.600,0.800) {$ a_{2} $};
\node[qvertex] (r4v2) at (8.600,-0.800) {$ c_{1} $};
\node[qvertex] (r4v3) at (10.600,-0.800) {$ c_{2} $};
\node[qvertex] (r4v4) at (7.600,0.000) {$ b_{1} $};
\node[qvertex,mutvertex] (r4v5) at (8.600,0.000) {$ b_{2} $};
\node[qvertex] (r4v6) at (9.600,0.000) {$ b_{3} $};
\node[qvertex,mutvertex] (r4v7) at (10.600,0.000) {$ b_{4} $};
\node[qvertex] (r4v8) at (11.600,0.000) {$ b_{5} $};
\draw[qarrow] (r4v0) to (r4v1);
\draw[qarrow] (r4v1) to (r4v6);
\draw[qarrow] (r4v1) to (r4v8);
\draw[qarrow] (r4v2) to (r4v4);
\draw[qarrow] (r4v2) to (r4v6);
\draw[qarrow] (r4v3) to (r4v2);
\draw[qarrow] (r4v4) to (r4v5);
\draw[qarrow] (r4v5) to (r4v0);
\draw[qarrow] (r4v5) to (r4v2);
\draw[qarrow] (r4v6) to (r4v5);
\draw[qarrow] (r4v6) to (r4v7);
\draw[qarrow] (r4v7) to (r4v1);
\draw[qarrow] (r4v7) to (r4v3);
\draw[qarrow] (r4v8) to (r4v7);
\draw[transition] (3,-0.750)--(3,-1.750);
\node[anchor=west,font=\scriptsize] at (3.15,-1.250) {2 flips with the same label};
\draw[transition] (11.950,-0.950) .. controls (12.200,-1.100) and (12.200,-1.400) .. (11.950,-1.550);
\end{scope}
\begin{scope}[yshift=-12.500cm]
\draw (0,0.6)--(6,0.6)--(6,-0.6)--(0,-0.6)--cycle;
\draw[diffdiag] (1,0.6)--(0,-0.6);
\draw (1,0.6)--(1,-0.6);
\draw[diffdiag] (2,0.6)--(1,-0.6);
\draw (2,0.6)--(2,-0.6);
\draw[diffdiag] (3,0.6)--(2,-0.6);
\draw (3,0.6)--(3,-0.6);
\draw[diffdiag] (4,0.6)--(3,-0.6);
\draw (4,0.6)--(4,-0.6);
\draw[diffdiag] (5,0.6)--(4,-0.6);
\draw (5,0.6)--(5,-0.6);
\draw[diffdiag] (6,0.6)--(5,-0.6);
\node[inner sep=0pt,font=\footnotesize] at (0.319,0.281) {$ 2 $};
\node[inner sep=0pt,font=\footnotesize] at (0.681,-0.281) {$ 1 $};
\node[inner sep=0pt,font=\footnotesize] at (1.319,0.281) {$ 3 $};
\node[inner sep=0pt,font=\footnotesize] at (1.681,-0.281) {$ 2 $};
\node[inner sep=0pt,font=\footnotesize] at (2.319,0.281) {$ 2 $};
\node[inner sep=0pt,font=\footnotesize] at (2.681,-0.281) {$ 3 $};
\node[inner sep=0pt,font=\footnotesize] at (3.319,0.281) {$ 1 $};
\node[inner sep=0pt,font=\footnotesize] at (3.681,-0.281) {$ 2 $};
\node[inner sep=0pt,font=\footnotesize] at (4.319,0.281) {$ 2 $};
\node[inner sep=0pt,font=\footnotesize] at (4.681,-0.281) {$ 1 $};
\node[inner sep=0pt,font=\footnotesize] at (5.319,0.281) {$ 3 $};
\node[inner sep=0pt,font=\footnotesize] at (5.681,-0.281) {$ 2 $};
\fill (0,0.6) circle (.65pt); \fill (0,-0.6) circle (.65pt);
\fill (1,0.6) circle (.65pt); \fill (1,-0.6) circle (.65pt);
\fill (2,0.6) circle (.65pt); \fill (2,-0.6) circle (.65pt);
\fill (3,0.6) circle (.65pt); \fill (3,-0.6) circle (.65pt);
\fill (4,0.6) circle (.65pt); \fill (4,-0.6) circle (.65pt);
\fill (5,0.6) circle (.65pt); \fill (5,-0.6) circle (.65pt);
\fill (6,0.6) circle (.65pt); \fill (6,-0.6) circle (.65pt);
\node[anchor=west] at (12.150,0) {$ \mathcal Q^{\mathrm{op}} $};
\node[qvertex] (r5v0) at (8.600,0.800) {$ a_{1} $};
\node[qvertex] (r5v1) at (10.600,0.800) {$ a_{2} $};
\node[qvertex] (r5v2) at (8.600,-0.800) {$ c_{1} $};
\node[qvertex] (r5v3) at (10.600,-0.800) {$ c_{2} $};
\node[qvertex] (r5v4) at (7.600,0.000) {$ b_{1} $};
\node[qvertex] (r5v5) at (8.600,0.000) {$ b_{2} $};
\node[qvertex] (r5v6) at (9.600,0.000) {$ b_{3} $};
\node[qvertex] (r5v7) at (10.600,0.000) {$ b_{4} $};
\node[qvertex] (r5v8) at (11.600,0.000) {$ b_{5} $};
\draw[qarrow] (r5v0) to (r5v1);
\draw[qarrow] (r5v0) to (r5v5);
\draw[qarrow] (r5v1) to (r5v7);
\draw[qarrow] (r5v2) to (r5v5);
\draw[qarrow] (r5v3) to (r5v2);
\draw[qarrow] (r5v3) to (r5v7);
\draw[qarrow] (r5v4) to (r5v0);
\draw[qarrow] (r5v5) to (r5v4);
\draw[qarrow] (r5v5) to (r5v6);
\draw[qarrow] (r5v6) to (r5v0);
\draw[qarrow] (r5v6) to (r5v3);
\draw[qarrow] (r5v7) to (r5v6);
\draw[qarrow] (r5v7) to (r5v8);
\draw[qarrow] (r5v8) to (r5v3);
\draw[transition] (3,-0.750)--(3,-1.750);
\node[anchor=west,font=\scriptsize] at (3.15,-1.250) {6 flips with different labels};
\draw[transition] (11.950,-0.950) .. controls (12.200,-1.100) and (12.200,-1.400) .. (11.950,-1.550);
\end{scope}
\begin{scope}[yshift=-15.000cm]
\fill[cornertri] (0,-0.6)--(1,-0.6)--(0,0.6);
\fill[cornertri] (5,0.6)--(6,0.6)--(6,-0.6);
\draw (0,0.6)--(6,0.6)--(6,-0.6)--(0,-0.6)--cycle;
\draw (0,0.6)--(1,-0.6);
\draw (1,0.6)--(1,-0.6);
\draw (1,0.6)--(2,-0.6);
\draw (2,0.6)--(2,-0.6);
\draw (2,0.6)--(3,-0.6);
\draw (3,0.6)--(3,-0.6);
\draw (3,0.6)--(4,-0.6);
\draw (4,0.6)--(4,-0.6);
\draw (4,0.6)--(5,-0.6);
\draw (5,0.6)--(5,-0.6);
\draw (5,0.6)--(6,-0.6);
\node[inner sep=0pt,font=\footnotesize] at (0.319,-0.281) {$ 1 $};
\node[inner sep=0pt,font=\footnotesize] at (0.681,0.281) {$ 2 $};
\node[inner sep=0pt,font=\footnotesize] at (1.319,-0.281) {$ 2 $};
\node[inner sep=0pt,font=\footnotesize] at (1.681,0.281) {$ 3 $};
\node[inner sep=0pt,font=\footnotesize] at (2.319,-0.281) {$ 3 $};
\node[inner sep=0pt,font=\footnotesize] at (2.681,0.281) {$ 2 $};
\node[inner sep=0pt,font=\footnotesize] at (3.319,-0.281) {$ 2 $};
\node[inner sep=0pt,font=\footnotesize] at (3.681,0.281) {$ 1 $};
\node[inner sep=0pt,font=\footnotesize] at (4.319,-0.281) {$ 1 $};
\node[inner sep=0pt,font=\footnotesize] at (4.681,0.281) {$ 2 $};
\node[inner sep=0pt,font=\footnotesize] at (5.319,-0.281) {$ 2 $};
\node[inner sep=0pt,font=\footnotesize] at (5.681,0.281) {$ 3 $};
\fill (0,0.6) circle (.65pt); \fill (0,-0.6) circle (.65pt);
\fill (1,0.6) circle (.65pt); \fill (1,-0.6) circle (.65pt);
\fill (2,0.6) circle (.65pt); \fill (2,-0.6) circle (.65pt);
\fill (3,0.6) circle (.65pt); \fill (3,-0.6) circle (.65pt);
\fill (4,0.6) circle (.65pt); \fill (4,-0.6) circle (.65pt);
\fill (5,0.6) circle (.65pt); \fill (5,-0.6) circle (.65pt);
\fill (6,0.6) circle (.65pt); \fill (6,-0.6) circle (.65pt);
\node[anchor=west] at (12.150,0) {$ \mathcal Q^{\mathrm{op}} $};
\node[qvertex] (r6v0) at (8.600,0.800) {$ a_{1} $};
\node[qvertex] (r6v1) at (10.600,0.800) {$ a_{2} $};
\node[qvertex] (r6v2) at (8.600,-0.800) {$ c_{1} $};
\node[qvertex] (r6v3) at (10.600,-0.800) {$ c_{2} $};
\node[qvertex] (r6v4) at (7.600,0.000) {$ b_{1} $};
\node[qvertex] (r6v5) at (8.600,0.000) {$ b_{2} $};
\node[qvertex] (r6v6) at (9.600,0.000) {$ b_{3} $};
\node[qvertex] (r6v7) at (10.600,0.000) {$ b_{4} $};
\node[qvertex] (r6v8) at (11.600,0.000) {$ b_{5} $};
\draw[qarrow] (r6v0) to (r6v1);
\draw[qarrow] (r6v0) to (r6v5);
\draw[qarrow] (r6v1) to (r6v7);
\draw[qarrow] (r6v2) to (r6v5);
\draw[qarrow] (r6v3) to (r6v2);
\draw[qarrow] (r6v3) to (r6v7);
\draw[qarrow] (r6v4) to (r6v0);
\draw[qarrow] (r6v5) to (r6v4);
\draw[qarrow] (r6v5) to (r6v6);
\draw[qarrow] (r6v6) to (r6v0);
\draw[qarrow] (r6v6) to (r6v3);
\draw[qarrow] (r6v7) to (r6v6);
\draw[qarrow] (r6v7) to (r6v8);
\draw[qarrow] (r6v8) to (r6v3);
\draw[transition] (3,-0.750)--(3,-1.750);
\node[anchor=west,font=\scriptsize] at (3.15,-1.250) {left and right reflections};
\draw[transition] (11.950,-0.950) .. controls (12.200,-1.100) and (12.200,-1.400) .. (11.950,-1.550);
\end{scope}
\begin{scope}[yshift=-17.500cm]
\draw (0,0.6)--(6,0.6)--(6,-0.6)--(0,-0.6)--cycle;
\draw (1,0.6)--(0,-0.6);
\draw[samediag] (2,0.6)--(0,-0.6);
\draw (2,0.6)--(1,-0.6);
\draw[samediag] (3,0.6)--(1,-0.6);
\draw (3,0.6)--(2,-0.6);
\draw[samediag] (4,0.6)--(2,-0.6);
\draw (4,0.6)--(3,-0.6);
\draw[samediag] (5,0.6)--(3,-0.6);
\draw (5,0.6)--(4,-0.6);
\draw[samediag] (6,0.6)--(4,-0.6);
\draw (6,0.6)--(5,-0.6);
\node[inner sep=0pt,font=\footnotesize] at (0.319,0.281) {$ 1 $};
\node[inner sep=0pt,font=\footnotesize] at (1.115,0.355) {$ 2 $};
\node[inner sep=0pt,font=\footnotesize] at (0.885,-0.355) {$ 2 $};
\node[inner sep=0pt,font=\footnotesize] at (2.115,0.355) {$ 3 $};
\node[inner sep=0pt,font=\footnotesize] at (1.885,-0.355) {$ 3 $};
\node[inner sep=0pt,font=\footnotesize] at (3.115,0.355) {$ 2 $};
\node[inner sep=0pt,font=\footnotesize] at (2.885,-0.355) {$ 2 $};
\node[inner sep=0pt,font=\footnotesize] at (4.115,0.355) {$ 1 $};
\node[inner sep=0pt,font=\footnotesize] at (3.885,-0.355) {$ 1 $};
\node[inner sep=0pt,font=\footnotesize] at (5.115,0.355) {$ 2 $};
\node[inner sep=0pt,font=\footnotesize] at (4.885,-0.355) {$ 2 $};
\node[inner sep=0pt,font=\footnotesize] at (5.681,-0.281) {$ 3 $};
\fill (0,0.6) circle (.65pt); \fill (0,-0.6) circle (.65pt);
\fill (1,0.6) circle (.65pt); \fill (1,-0.6) circle (.65pt);
\fill (2,0.6) circle (.65pt); \fill (2,-0.6) circle (.65pt);
\fill (3,0.6) circle (.65pt); \fill (3,-0.6) circle (.65pt);
\fill (4,0.6) circle (.65pt); \fill (4,-0.6) circle (.65pt);
\fill (5,0.6) circle (.65pt); \fill (5,-0.6) circle (.65pt);
\fill (6,0.6) circle (.65pt); \fill (6,-0.6) circle (.65pt);
\node[anchor=west] at (12.150,0) {$ \mathcal Q^{\mathrm{op}} $};
\node[qvertex] (r7v0) at (8.600,0.800) {$ a_{1} $};
\node[qvertex,mutvertex] (r7v1) at (10.600,0.800) {$ a_{2} $};
\node[qvertex,mutvertex] (r7v2) at (8.600,-0.800) {$ c_{1} $};
\node[qvertex] (r7v3) at (10.600,-0.800) {$ c_{2} $};
\node[qvertex,mutvertex] (r7v4) at (7.600,0.000) {$ b_{1} $};
\node[qvertex] (r7v5) at (8.600,0.000) {$ b_{2} $};
\node[qvertex,mutvertex] (r7v6) at (9.600,0.000) {$ b_{3} $};
\node[qvertex] (r7v7) at (10.600,0.000) {$ b_{4} $};
\node[qvertex,mutvertex] (r7v8) at (11.600,0.000) {$ b_{5} $};
\draw[qarrow] (r7v0) to (r7v1);
\draw[qarrow] (r7v0) to (r7v5);
\draw[qarrow] (r7v1) to (r7v7);
\draw[qarrow] (r7v2) to (r7v5);
\draw[qarrow] (r7v3) to (r7v2);
\draw[qarrow] (r7v3) to (r7v7);
\draw[qarrow] (r7v4) to (r7v0);
\draw[qarrow] (r7v5) to (r7v4);
\draw[qarrow] (r7v5) to (r7v6);
\draw[qarrow] (r7v6) to (r7v0);
\draw[qarrow] (r7v6) to (r7v3);
\draw[qarrow] (r7v7) to (r7v6);
\draw[qarrow] (r7v7) to (r7v8);
\draw[qarrow] (r7v8) to (r7v3);
\draw[transition] (3,-0.750)--(3,-1.750);
\node[anchor=west,font=\scriptsize] at (3.15,-1.250) {5 flips with the same label};
\draw[transition] (11.950,-0.950) .. controls (12.200,-1.100) and (12.200,-1.400) .. (11.950,-1.550);
\end{scope}
\begin{scope}[yshift=-20.000cm]
\draw (0,0.6)--(6,0.6)--(6,-0.6)--(0,-0.6)--cycle;
\draw (1,0.6)--(0,-0.6);
\draw (1,0.6)--(1,-0.6);
\draw (2,0.6)--(1,-0.6);
\draw (2,0.6)--(2,-0.6);
\draw (3,0.6)--(2,-0.6);
\draw (3,0.6)--(3,-0.6);
\draw (4,0.6)--(3,-0.6);
\draw (4,0.6)--(4,-0.6);
\draw (5,0.6)--(4,-0.6);
\draw (5,0.6)--(5,-0.6);
\draw (6,0.6)--(5,-0.6);
\node[inner sep=0pt,font=\footnotesize] at (0.319,0.281) {$ 1 $};
\node[inner sep=0pt,font=\footnotesize] at (0.681,-0.281) {$ 2 $};
\node[inner sep=0pt,font=\footnotesize] at (1.319,0.281) {$ 2 $};
\node[inner sep=0pt,font=\footnotesize] at (1.681,-0.281) {$ 3 $};
\node[inner sep=0pt,font=\footnotesize] at (2.319,0.281) {$ 3 $};
\node[inner sep=0pt,font=\footnotesize] at (2.681,-0.281) {$ 2 $};
\node[inner sep=0pt,font=\footnotesize] at (3.319,0.281) {$ 2 $};
\node[inner sep=0pt,font=\footnotesize] at (3.681,-0.281) {$ 1 $};
\node[inner sep=0pt,font=\footnotesize] at (4.319,0.281) {$ 1 $};
\node[inner sep=0pt,font=\footnotesize] at (4.681,-0.281) {$ 2 $};
\node[inner sep=0pt,font=\footnotesize] at (5.319,0.281) {$ 2 $};
\node[inner sep=0pt,font=\footnotesize] at (5.681,-0.281) {$ 3 $};
\fill (0,0.6) circle (.65pt); \fill (0,-0.6) circle (.65pt);
\fill (1,0.6) circle (.65pt); \fill (1,-0.6) circle (.65pt);
\fill (2,0.6) circle (.65pt); \fill (2,-0.6) circle (.65pt);
\fill (3,0.6) circle (.65pt); \fill (3,-0.6) circle (.65pt);
\fill (4,0.6) circle (.65pt); \fill (4,-0.6) circle (.65pt);
\fill (5,0.6) circle (.65pt); \fill (5,-0.6) circle (.65pt);
\fill (6,0.6) circle (.65pt); \fill (6,-0.6) circle (.65pt);
\node[anchor=west] at (12.150,0) {$ \mathcal Q $};
\node[qvertex] (r8v0) at (8.600,0.800) {$ a_{1} $};
\node[qvertex] (r8v1) at (10.600,0.800) {$ a_{2} $};
\node[qvertex] (r8v2) at (8.600,-0.800) {$ c_{1} $};
\node[qvertex] (r8v3) at (10.600,-0.800) {$ c_{2} $};
\node[qvertex] (r8v4) at (7.600,0.000) {$ b_{1} $};
\node[qvertex] (r8v5) at (8.600,0.000) {$ b_{2} $};
\node[qvertex] (r8v6) at (9.600,0.000) {$ b_{3} $};
\node[qvertex] (r8v7) at (10.600,0.000) {$ b_{4} $};
\node[qvertex] (r8v8) at (11.600,0.000) {$ b_{5} $};
\draw[qarrow] (r8v0) to (r8v4);
\draw[qarrow] (r8v0) to (r8v6);
\draw[qarrow] (r8v1) to (r8v0);
\draw[qarrow] (r8v2) to (r8v3);
\draw[qarrow] (r8v3) to (r8v6);
\draw[qarrow] (r8v3) to (r8v8);
\draw[qarrow] (r8v4) to (r8v5);
\draw[qarrow] (r8v5) to (r8v0);
\draw[qarrow] (r8v5) to (r8v2);
\draw[qarrow] (r8v6) to (r8v5);
\draw[qarrow] (r8v6) to (r8v7);
\draw[qarrow] (r8v7) to (r8v1);
\draw[qarrow] (r8v7) to (r8v3);
\draw[qarrow] (r8v8) to (r8v7);
\end{scope}
\end{tikzpicture}
\caption{The sequence of flips and mutations corresponding to the rotation \(\rot\). It gives the time evolution of the level \(\ell = 3\) restricted \(Y\)-system of type \(B_2^{(1)}\).
 The right column shows the quiver of each seed, where \(a_j=(1,j)\), \(b_j=(2,j)\), and \(c_j=(3,j)\).}
\label{fig:appendix-b2-cluster}
\end{figure}

For each \(i\in I\), let \(d_i\) be the size of its \(\theta\)-orbit, and set \(H=d\ell\).
Choose integers \(\eta_i\) such that
\[
 0\le\eta_i<2d_i,\qquad
 \eta_{\theta(i)}\equiv\eta_i+2\pmod{2d_i}.
\]
Moreover, for adjacent vertices \(i,j\), we require that
\(\eta_j\equiv\eta_i+d_i\pmod{2d_i}\) if \(d_i=d_j\), and
that \(\eta_i,\eta_j\) have different parities if \(d_i\ne d_j\).
Such integers can be chosen using the tree structure of the Dynkin diagram.
In this section, we use the positive braid of the \(\theta\)-Coxeter element corresponding to this choice,
\[
 c=\prod_{\eta_i=0}s_i\;\cdot\!\prod_{\eta_i=1}s_i.
\]
The factors within each product commute, so this does not depend on their order.

Using two copies \(I^\pm=\{i^\pm:i\in I\}\) of \(I\),
we call a finite word in \(I^-\sqcup I^+\) a \emph{double word}.
We define the initial word \(\mathbf{i}\) by arranging, in increasing order of \(x\), the pairs \((x,i^\pm)\) of an integer \(x\) and a letter that satisfy
\begin{equation}
 \begin{aligned}
 (x,i^-):&\quad 0\le x<H,&&x\equiv\eta_i\pmod{2d_i},\\
 (x,i^+):&\quad 0<x\le H,&&x\equiv-\eta_i\pmod{2d_i},
 \end{aligned}
 \label{eq:cluster-double-word}
\end{equation}
and concatenating their letter parts.
For the same \(x\), positive letters are placed first, and the order of mutually commuting letters of the same sign is fixed.
Let \(p,q\) be the positive braids represented by the negative and positive subwords of \(\mathbf{i}\), respectively.
Letting \(q^\circ\) be the positive braid obtained by reversing the word of \(q\), we have \(pq^\circ=\beta^\ell\).
Therefore the composition of left and right reflection maps gives an equivalence \(\Conf^{\beta^\ell}_\varnothing\simeq\Conf^p_q\).
In the triangulation corresponding to \(\mathbf{i}\) \cite[Definition~3.1]{ShenWeng2021},
triangles with the same label appear alternately pointing upward and downward.

The mutable vertices of the seed of this triangulation are the closed strings between occurrences number \(h\) and \(h+1\) of the label \(i\),
denoted by \((i,h)\). Let \(V\) be the set of them, and for \(r=0,1\), define \(M_r\subset V\) by
\begin{equation}
 V=\{(i,h):i\in I,\ 1\le h<H/d_i\},\qquad
 M_r=\{(i,h)\in V:\eta_i+d_i h\equiv r\pmod{2d_i}\}.
 \label{eq:cluster-mutation-sets}
\end{equation}
The mutations within each \(M_r\) commute, and we write \(\mu_{M_r}\) for their composite.
Using Shen--Weng's description of coordinate transformations under flips and reflection maps
\cite[Propositions~3.7, 3.17, 4.2, 4.11]{ShenWeng2021},
we obtain the following expression for the rotation in cluster Poisson coordinates:
\begin{equation}
 \rot=\mu_{M_0}\circ\mu_{M_1}\circ\theta.
 \label{eq:cluster-rotation}
\end{equation}

In particular, when \(\Aff_\theta(\Dcal^\vee)^\vee\) is one of
\(B_r^{(1)},C_r^{(1)},F_4^{(1)},G_2^{(1)}\),
the map \(\rot^{-1}\) corresponds to the time evolution
\(u\mapsto u+2/d\) of the level \(\ell\) restricted \(Y\)-system of
\cite[\S2]{KunibaNakanishi1992}.
Therefore, Corollary~\ref{cor:appendix-periodicity} gives
an alternative geometric proof of the \(2(\ell+h^\vee)\)-periodicity of the level restricted \(Y\)-systems \cite{IIKKN2013B,IIKKN2013CFG}.
Figure~\ref{fig:appendix-b2-cluster} shows an example of type \(B_2^{(1)}\) for \(\Dcal=A_3\), \(\theta=(1\,3)\), \(\ell=3\),
with the choice \((\eta_1,\eta_2,\eta_3)=(0,1,2)\).
In this case, \(c=12\), \(\beta=1232\), \(p=123212\), and \(q=232123\).

\printbibliography

\bigskip

\noindent
\textsc{Yuma Mizuno, School of Mathematical Sciences, University College Cork, Western Road,
Cork, Ireland.}\par
\noindent
\textit{Email address}: \texttt{YMizuno@ucc.ie}

\end{document}